\documentclass[11pt]{article}

\usepackage{amssymb, graphicx}
\usepackage[bookmarksnumbered, colorlinks, plainpages]{hyperref}
\hypersetup{colorlinks=true,linkcolor=cyan, citecolor=cyan, urlcolor=red, pdftoolbar=true}
\usepackage[centertags]{amsmath}
\usepackage{mathrsfs}
\usepackage{amsthm}
\allowdisplaybreaks[4]
\usepackage[top=2cm,bottom=2cm,left=3cm,right=3cm,includehead,includefoot]{geometry}
\usepackage{indentfirst}
 \newtheorem{thm}{Theorem}[section]
 
 \newtheorem{lem}[thm]{Lemma}
 
 \theoremstyle{definition}
 
 \theoremstyle{remark}

 \numberwithin{equation}{section}
\theoremstyle{plain}
\newtheorem{remark}{Remark}

\begin{document}
\title{{ Compact linear combination of composition operators on weighted Bergman-Orlicz spaces induced by doubling weights \thanks{ This work is supported by the National Natural Science Foundation of China (12171138,12371129)} }
\author{Yucheng Li$^{a}$,\, Zhiyu Wang$^{a}$\thanks{Corresponding author: wangzhiyumath@163.com} and Liankuo Zhao$^{b}$ \\
{\small \texttt{liyucheng@hebtu.edu.cn,\,\,wangzhiyumath@163.com,\,\,lkzhao@sxnu.edu.cn}}\\
  \small {$^{a}$\,\,\,\,School of Mathematical Sciences, Hebei Normal University, Shijiazhuang 050024,  China  }\\
  \small {Hebei Key Laboratory of Computational Mathematics and Applications } \\
\small {$^{b}$\,\,\,\,School of Mathematical Sciences, Shanxi  Normal University, Taiyuan 030031,  China}}}

\date{}
\maketitle

\textbf{Abstract}\    Let  $A^{\Psi}_{\omega}(\mathbb{D})$ be the weighted Bergman-Orlicz spaces induced by doubling weights $\omega$ on the unit disk $\mathbb D$. Suppose  Orlicz function $\Psi$ satisfies the $\Delta_{2}$-condition. In this paper, the compactness of finite linear combination of composition operators on  $A^{\Psi}_{\omega}(\mathbb{D})$ are completely characterized in terms of Julia-Carathéodory type conditions and vanishing Carleson measure, respectively. 



\section{Introduction}

Let $\mathbb{D}$ be the open unit disk of the complex plane $\mathbb{C}$ and $H(\mathbb{D})$ denote the spaces of all holomorphic functions on $\mathbb{D}$.
An Orlicz function is a nonnegative convex function $\Psi:[0,\infty)\to[0,\infty)$ satisfying $\Psi(0)=0$ and $\Psi(t)>0$ for all $t>0$.
We say that $\Psi$ satisfies the $ \Delta_{2}$-condition if there exists a constant $C>1$ such that 
\begin{align} \label{eq: condition}            
  \Psi (2t) & \leq C \Psi (t), \qquad t>0.
\end{align}
 An Orlicz function $ \Psi$  is said to be of positive upper type  $p\, (p>0)$(see \cite{YLK}), if there exists a constant $C>0$  such that
\begin{align}\label{eq: positive upper type}
  \Psi(st) & \leq Cs^{p}\Psi(t), \qquad t>0, \,\,\,s\geq1.
\end{align}
A direct calculation shows that \eqref{eq: condition} and \eqref{eq: positive upper type} equivalent.
We also recall the following observation concerning the $\Delta_{2}$-condition (see Remarks 2(a) following Theorem 4.11 in {\rm\cite{LLQR2}}).

\begin{remark}
Let $\Psi$ be an Orlicz function satisfying the $ \Delta_{2}$-condition. Then for every $B>0$, there exists a constant $C=C_{B}>0$ such that
\begin{align}\label{eq: remark 1}
   \frac{1}{C}\Psi(t) & \leq\Psi(Bt)\leq C \Psi(t), \qquad t>0.
\end{align}
\end{remark}

Let $\omega$  be a positive integrable weight function on $\mathbb{D}$ and $ {\rm d}A(z)=\frac{1}{\pi}{\rm d}x{\rm d}y$ be the normalized area measure on $\mathbb{D}$. The weighted Bergman-Orlicz space $A^{\Psi}_{\omega}(\mathbb{D})$ consists of all $f\in H(\mathbb{D})$ such that there exists a constant $C>0$ satisfying
$$\int _ {\mathbb{D}} \Psi \bigg( \frac { | f (z) | } { C } \bigg) \omega(z) {\rm d}A (z) < + \infty.$$ 
It is equipped with the Luxemburg norm  
$$\|f\|_{A^{\Psi}_{\omega}(\mathbb{D})}=\inf\bigg\{C>0: \int _ { \mathbb{D} } \Psi \bigg( \frac { | f (z) | } { C } \bigg) \omega(z) {\rm d} A (z) \leq1\bigg\}.$$ 
Note that when $\Psi(t)=t^p$ for $1\leq p<\infty$, the space $A^{\Psi}_{\omega}(\mathbb{D})$ coincides with the weighted Bergman space $A^{p}(\omega)$. In particular, for $\omega(z)=(\alpha+1)(1-|z|^2)^\alpha$ with $\alpha>-1$, we write $A_\alpha^p(\mathbb{D})$ for $A^{p}(\omega)$. For further forms of Orlicz functions $\Psi$, we refer the readers to \cite{LLQR2,RR} .
We next recall the weight classes considered in this paper.

A weight $\omega:\mathbb{D}\to[0,\infty)$ is called radial if $\omega(z)=\omega(|z|)$ for all $z\in\mathbb{D}$. For a radial weight, set 
\begin{align}\label{eq:omega-hat}
  \widehat{\omega}(r) & =\int_r^1 \omega(s)\,ds, \qquad 0\leq r<1, 
\end{align}
and assume $\int_r^1 \omega(s)\,ds>0$ for any $0\leq r<1$.
A radial weight $\omega\in\widehat{\mathcal{D}}$ if there exists a constant $C=C(\omega)>1$ such that
\begin{align*}
  \widehat{\omega}(r) & \leq C\widehat{\omega}\left(\frac{1+r}{2}\right), \qquad 0\leq r<1.
\end{align*}
Moreover, we say that $\omega\in\check{\mathcal{D}}$ if there exist $K=K(\omega)>1$ and $C=C(\omega)>1$ such that
\begin{align*}
  \widehat{\omega}(r) & \geq C\widehat{\omega}\left(1-\frac{1-r}{K}\right), \qquad 0\leq r<1.
\end{align*}
Let $\mathcal{D}=\widehat{\mathcal{D}}\cap\check{\mathcal{D}}$.
We say that $\omega$ is a doubling weight if $\omega\in \mathcal{D}$.
For more details on $\widehat{\mathcal{D}}$,  $\check{\mathcal{D}}$ and the operator theory on weighted Bergman spaces induced by these classes of weights, see \cite{LRW,PJ,PR,PR2,PRS}.

Let $S(\mathbb{D})$ be the set of all holomorphic self-maps on $\mathbb{D}$. For $\varphi\in S(\mathbb{D})$, the composition operator
is defined by 
 $$C_{\varphi}f=f\circ\varphi,   \qquad f\in H(\mathbb{D}).$$
Composition operators have been extensively studied in a variety of function spaces (see \cite{Ch,Ch1,CM,GT,LLQR,Sh,Sh1,Sm,SY}).
By Littlewood’s subordination principle, every composition operator is bounded on the standard weighted Bergman spaces and Bergman–Orlicz spaces.
In the weighted Bergman spaces, MacCluer and Shapiro \cite{MS} characterized compactness of composition operators by the Julia-Carathéodory type condition
\begin{align*}
R _ { \varphi } (z)  = \frac { 1 - | z | ^{2} } { 1 - | \varphi (z) | ^{2} } \rightarrow 0  \ \text{ as }\  | z | \to 1.
\end{align*}
For differences of composition operators on $A_{\alpha}^{2}(\mathbb{D})$, Moorhouse \cite{Mo} obtained a characterization involving
\begin{align*}
M _ { \varphi , \psi } (z)  = [ R _ { \varphi }(z) + R _ { \psi } (z) ] \rho _ { \varphi , \psi } (z) \to 0  \qquad    as  \,\,\,\,\,    | z | \to 1 ,
\end{align*}
where $\rho _ { \varphi , \psi } (z)= \rho(\varphi(z),\psi(z))$ and $\rho$ denotes the pseudohyperbolic distance.
Peláez and Rättyä \cite{PR1} further show that every composition operator is also bounded on weighted Bergman $A^p(\omega)$ for $\omega\in\widehat{\mathcal D}$.
Choe, Koo and Wang \cite{CKW}  characterized  the compact linear combination of composition operators with arbitrary symbols on $A_{\alpha}^{p}(\mathbb{D})$. Recently, Guo and Wang \cite{GW}
obtained a characterization of the compactness of linear combinations of composition operators on $A^p(\omega)$ for $\omega\in\mathcal D$. Motivated by these results, in this paper we further investigate the compactness of linear combinations of composition operators on weighted Bergman-Orlicz spaces $A^{\Psi}_{\omega}(\mathbb{D})$ with $\omega\in\mathcal D$.
However, the results on linear combinations of composition operators mentioned above concern weighted Bergman spaces, whose norms are based on power functions. In weighted Bergman--Orlicz spaces, these power functions are replaced by a more general Orlicz function, so the previous compactness criteria do not directly apply.

To state our main results, we introduce some notations.
Let $N \geq 2$ be an arbitrary integer, and let $\mathscr{P}_N$ denote the group of all permutations of the index set $$\Lambda _ { N } = \{ 1 , 2 , \ldots , N \} .$$
Assume that the coefficients
$$a _ { 1 } , \ldots , a _ { N } \in \mathbb { C } \backslash \{ 0 \}$$
and the symbol functions
$$\varphi _ { 1 } , \varphi _ { 2 } , \ldots, \varphi _ { N } \in S(\mathbb{D}).$$
Let
$$T _ { i }  = C _ { \varphi _ { i } }  ,   \,\,\,     T _ { i , k }  = T _ { i } - T _ { k}$$
for $i,k=1,\ldots,N$. For simplicity, denote the linear combination operators by
$$T = \sum _ { i= 1 } ^{N} a _ { i } T _ { i }.$$
We identify each $\eta \in \mathscr{P}_N$ with the ordered $N$-tuple $(\eta_1, \ldots, \eta_N)$, where $\eta_i = \eta(i)$ and write $$\eta = ( \eta _ { 1 } , \ldots , \eta _ { N } ) .$$
For $\eta \in \mathscr{P}_N$, let
$$s _ { i } ^{\eta} = s _ { i } ^{\eta} ( a _ { 1 } , \cdot\cdot\cdot , a _ { N } )  = \sum _ { k = 1 } ^{i} a _ { \eta _ { k } }$$
for all $i \in \Lambda _ { N }$. The operator $T$  can be rewritten as
\begin{align}\label{eq: operator-T}
T = \sum _ { i = 1 } ^{N - 1} s _ { i } ^{\eta} T _ { \eta _ { i } , \eta _ { i + 1 } } + s _ { N } ^{\eta} T _ { \eta _ { N } }
\end{align}
for each $ \eta\in \mathscr{P}_N$.
Furthermore, let
$$R _ { i }  = R _ { \varphi _ { i } } , \,\,\,   M _ { i , k }  = M _ { \varphi _ { i } , \varphi _ { k } } ,  \,\,\,   \rho _ { i , k }  = \rho _ { \varphi _ { i } , \varphi _ { k } }$$
for $i,k\in\Lambda_N$ and $\varphi _ { i } , \varphi _ { k}\in S(\mathbb{D})$.
By the Schwarz-Pick lemma, each $R_i$ is bounded on $\mathbb{D}$. Consequently, each $M_{i,k}$ is also bounded on $\mathbb{D}$.

Define
\begin{align}\label{eq: define-Q}
Q _ { \eta}  = \sum _ { i = 1 } ^{N - 1} | s _ { i } ^{\eta} |  M _ { \eta _ { i } , \eta _ { i + 1 } } + | s _ { N } ^{\eta} |  R _ { \eta _ { N } }
\end{align}
for $\eta\in\mathscr{P}_{N}$, and set
\begin{align*}
Q  = \prod _ { \eta \in \mathscr { P } _ { N } } Q _ { \eta }.
\end{align*}
Note that all functions  $Q_{\eta}$ are bounded on $\mathbb{D}$.

Let $$f_{z,\gamma}(\xi)=\bigg(\frac{1-|z|^{2}}{1-\bar{z}\xi}\bigg)^{\gamma+1}\Psi^{-1}\bigg(\frac{1}{\omega(S(z))}\bigg)$$ as defined in Lemma \ref{le: text-function}.

The following theorem is our first main result: Characterizing compact linear combination of composition operators via Julia-Carathéodory type conditions. 
\begin{thm}\label{th: Julia-Caratheodory-type}
Let $\Psi$ be an Orlicz function satisfying the $\Delta_{2}$-condition. Suppose that $\omega\in \mathcal{D}$, $\varphi_i\in S(\mathbb{D})$ and $a_i \in \mathbb{C} \backslash \{0\}$ for each $i\in\Lambda_N$. Then the following statements are equivalent:

{\rm (a)} $T$ is compact on $A^{\Psi}_{\omega}(\mathbb{D})$;

{\rm (b)}  $\lim\limits_{|z|\to1} \|Tf_{z,\gamma}\|_{A^{\Psi}_{\omega}(\mathbb{D})}=0$ for all $\gamma>\gamma_{0}$;

{\rm (c)} $\lim\limits_{|z|\to1}Q(z)=0$.
\end{thm}
Theorem 1.1 gives a characterization of the compactness of linear combinations of composition operators in terms of Julia-Carathéodory type conditions.
Carleson measures play an important role in the study of boundedness and compactness of composition operators. For a single composition operator, the change-of-variables formula provides a connection between the operator and its associated pullback measure(see \eqref{eq:change-of-variable}). Dong, Duan, Guo and Wang \cite{DDGW} investigate the Carleson measure for weighted Bergman-Orlicz spaces induced by doubling weight. We recall the definitions of Carleson measure on $A^{\Psi}_{\omega}(\mathbb{D})$.

For a positive Borel measure $\mu$ on $\mathbb{D}$, let $L^{\Psi}_{\mu}(\mathbb{D})$ be the space of all measurable functions $f$ on $\mathbb{D}$ for which
$$\int _ {\mathbb{D}} \Psi \bigg( \frac { | f (z) | } { C } \bigg)  {\rm d}\mu (z) < + \infty,$$ 
 for some $C>0$, equipped with Luxemburg norm
$$\|f\|_{L^{\Psi}_{\mu}(\mathbb{D})}=\inf\bigg\{C>0: \int _ { \mathbb{D} } \Psi \bigg( \frac { | f (z) | } { C } \bigg)  {\rm d}\mu (z) \leq1\bigg\}.$$
We say that $\mu$ is a Carleson measure for $A^{\Psi}_{\omega}(\mathbb{D})$,
if identity embedding $I_{d} :A^{\Psi}_{\omega}(\mathbb{D})\rightarrow L^{\Psi}_{\mu}(\mathbb{D})$  is bounded, that is, there exists a constant  $C>0$ such that
\begin{align*}
  \|f\|_{L^{\Psi}_{\mu}(\mathbb{D})} & \leq C \|f\|_{A^{\Psi}_{\omega}(\mathbb{D})}, 
\end{align*}
for any $f\in A^{\Psi}_{\omega}(\mathbb{D})$.

Moreover, we say that $\mu$ is a vanishing Carleson measure for $A^{\Psi}_{\omega}(\mathbb{D})$, if   
identity embedding $I_{d} :A^{\Psi}_{\omega}(\mathbb{D})\rightarrow L^{\Psi}_{\mu}(\mathbb{D})$  is compact, that is, for any bounded sequence $\{f_{n}\}$ on $A^{\Psi}_{\omega}(\mathbb{D})$, if $f_{n}\to0$ uniformly on every compact subset of $\mathbb{D}$, then $$\lim_{n\to\infty} \|f_{n}\|_{L^{\Psi}_{\mu}(\mathbb{D})}=0. $$

To study the compactness of linear combinations of composition operators, we introduce some notations. Given $\eta\in \mathscr{P}_N$, we write
\begin{align}\label{eq: define-G}
G _ { \eta }  = \left\{ z \in \mathbb { D } : Q _ { \eta } (z) = \min _ { \tau \in \mathscr { P } _ { N } } Q _ { \tau } (z) \right\}.
\end{align}
 For any Borel set $E \subset \mathbb{D}$, we define the joint pullback measures $\mu_{\Psi,\eta_i,\omega} $ and $\nu_{\Psi,\eta_i,\omega} $ as follows,
$$\mu_{\Psi,\eta_i,\omega}(E)  = \int  _{\varphi _ { \eta _ { i } } ^{- 1} (E)  \cap G _ { \eta } } \Psi(M_{\eta _ { i },\eta _ { i+1 }}) \omega{\rm d} A +\int  _{\varphi _ { \eta _ { i+1 } } ^{- 1} (E)  \cap G _ { \eta } }\Psi( M_{\eta _ { i },\eta _ { i+1 }}) \omega {\rm d} A, $$
$$\nu_{\Psi,\eta_i,\omega} (E)  = \int  _{\varphi _ { \eta _ { i } } ^{- 1} (E)  \cap G _ { \eta } } \Psi(\rho_{\eta _ { i },\eta _ { i+1 }})\omega {\rm d} A +\int  _{\varphi _ { \eta _ { i+1 } } ^{- 1} (E)  \cap G _ { \eta } } \Psi(\rho_{\eta _ { i },\eta _ { i+1 }})  \omega {\rm d} A $$
for $1 \leq i < N$ and
$$\mu _ { \Psi,\eta _ { N },\omega } (E)  = \nu _ { \Psi,\eta _ { N } ,\omega} (E)  = \int _ {   \varphi _ { \eta _ { N } } ^{- 1} (E) \cap G _ { \eta }} \Psi( R _ { \eta _ { N } } ) \omega {\rm d} A. $$
Set
\begin{align*}
\mu _{\Psi,\omega} = \sum _ { \eta \in \mathscr { P } _ { N } } \mu _{\Psi,\eta,\omega},  \quad \text {where} \quad \mu _{\Psi,\eta,\omega}  = \sum _ { i = 1 } ^{N} | s _ { i} ^{\eta} |  \mu _ { \Psi,\eta _ { i },\omega },
\end{align*}
and
$$\nu _{\Psi,\omega}  = \sum _ { \eta \in \mathscr { P } _ { N } } \nu  _{\Psi,\eta,\omega},   \quad \text {where}  \quad   \nu  _{\Psi,\eta,\omega}  = \sum _ { i = 1 } ^{N} | s _ { i } ^{\eta} |  \nu _ { \Psi,\eta _ { i } ,\omega} .$$

The following theorem gives a characterization of the compactness of linear combinations of composition operators in terms of vanishing Carleson measure.
\begin{thm}
Let $\Psi$ be an Orlicz function satisfying the $\Delta_{2}$-condition and $\omega\in \mathcal{D}$. Let $\varphi_i\in S(\mathbb{D})$ and $a_i \in \mathbb{C} \backslash \{0\}$ for $i\in\Lambda _ { N }$. Then the following statements are equivalent:

{\rm (a)} $T$ is compact on $A^{\Psi}_{\omega}(\mathbb{D})$;

{\rm (b)} $\mu_{\Psi,\omega}$ is a vanishing Carleson measure for $A^{\Psi}_{\omega}(\mathbb{D})$;

{\rm (c)} $\nu_{\Psi,\omega}$ is a vanishing Carleson measure for $A^{\Psi}_{\omega}(\mathbb{D})$.
\end{thm}

This paper is organized as follows. In Section 2, we introduce some basic facts, properties of doubling weight and lemmas on weighted Bergman-Orlicz space that will be used throughout the paper. 
Sections 3 and 4 are devoted to the proofs of Theorems 1.1 and 1.2, respectively.

Throughout the paper, the notation $A\lesssim B$ ($B\gtrsim A$) means that there exists an absolute positive constant $C$  such that $A\leq CB$. We write $A\approx B$ means that both $A\lesssim B$ and $B\lesssim A$ hold. We use the letter $C$ to denote various positive constants which may vary at each occurrence.

\section{Preliminaries}
In this section, we introduce some basic facts and lemmas that will be used in the sequel. 

Let $\mathbb{T}=\partial\mathbb{D}$ . For an arc $I\subset\mathbb{T}$, define the associated Carleson square by
$$S(I)=\left\{z\in\mathbb{D}:\frac{z}{|z|}\in I,\;1-|I|\le |z|<1\right\},$$
where $|I|$ denotes normalized arc length.
For $a\in\mathbb{D}\setminus\{0\}$, let
$$I_a=\left\{\xi\in\mathbb{T}:\left|\arg(a\overline{\xi})\right|\le\pi(1-|a|)\right\},$$
 and $S(0)=\mathbb{D}$. For simplicity write $S(a)=S(I_a)$.

Pseudohyperbolic disk will also play an important role in our arguments. Let $\rho(a,z)$ be the pseudohyperbolic distance between $a$ and $z$, and
$$\rho(a,z)=\left|\frac{a-z}{1-a\bar{z}}\right|, \ a, z\in\mathbb{D}.$$
Simple computation shows that
\begin{align}\label{eq: distance}
1-\rho^{2}(a,z)=\frac{(1-|a|^{2})(1-|z|^{2})}{|1-a\bar{z}|^{2}}.
\end{align}
For $0<r<1$ and $a\in\mathbb{D}$, let $E_{r}(a)$ denote the pseudohyperbolic disk centered at $a$ with radius $r$ and $$E_{r}(a) =\{z\in\mathbb{D}: \rho(a,z)<r\}.$$
It is well known that $E_{r}(a)$ is also a Euclidean disk and its center and radius are given by
\begin{align}\label{eq:Euclidean disk}
\frac{1-r^{2}}{1-|a|^{2}r^{2}}a \qquad & \text{and} \qquad \frac{1-|a|^{2}}{1-|a|^{2}r^{2}}r.
\end{align}
 respectively. For all $a,z\in\mathbb{D}$, by \eqref{eq: distance}, we have
\begin{align}\label{eq: ratio-estimate}
\frac{1-\rho(a,z)}{1+\rho(a,z)}\leq\frac{1-|a|}{1-|z|}\leq\frac{1+\rho(a,z)}{1-\rho(a,z)}.
\end{align}
For any fixed $0<r<1$, 
\begin{align}\label{eq: area-estimate}
  A(E_{r}(a)) & \approx(1-|a|^{2})^{2},      \quad  a\in \mathbb{D},
\end{align}
 and if $z\in E_{r}(a)$, we have
\begin{align}\label{eq: kernel-estimate}
 |1-\bar{a}z| \approx 1-|a|^{2}\approx  1-|z|^{2}.
\end{align}
The implicit constants in \eqref{eq: area-estimate}  and \eqref{eq: kernel-estimate} depend only on $r$. For further details on the pseudohyperbolic distance and pseudohyperbolic disk, see \cite[Sections 4.1 and 4.2]{Zh}.

The following lemma provides a useful criterion for the compactness of $T$ on $A_{\omega}^{\Psi}(\mathbb{D})$ in terms of sequential convergence. Its proof follows from the argument for a single composition operator in \cite[Proposition 3.11]{CM} with the same reasoning applied to a linear combination.
\begin{lem}\label{le:compactness}
Let $T$ be a linear combination of composition operators as described in Section 1. Then $T$ is compact on $A_{\omega}^{\Psi}(\mathbb{D})$ if and only if for every bounded sequence $\{f_{n}\}$ in $A_{\omega}^{\Psi}(\mathbb{D})$ that converges to $0$ uniformly on compact subsets of $\mathbb{D}$, we have $\|Tf_{n}\|_{A_{\omega}^{\Psi}(\mathbb{D})}\rightarrow0$ as $n\rightarrow \infty$.
\end{lem}

The following lemma gives uniformly bounded test functions for $A^{\Psi}_{\omega}(\mathbb{D})$.
\begin{lem}(\cite[Lemma 11]{DDGW})\label{le: text-function}
Let $\Psi$ be an Orlicz function satisfying the $\Delta_{2}$-condition and $\omega\in \widehat{\mathcal{D}}$.
Then there exists a constat $\gamma_{0}=\gamma_{0}(\omega)>0$ such that for any $\gamma\geq\gamma_{0}$ and each $z\in \mathbb{D}$, the function 
 $$f_{z,\gamma}(\xi)=\bigg(\frac{1-|z|^{2}}{1-\bar{z}\xi}\bigg)^{\gamma+1}\Psi^{-1}\bigg(\frac{1}{\omega(S(z))}\bigg),       \,\,\,\,\,\,   \xi\in\mathbb{D}$$ 
belongs to $A^{\Psi}_{\omega}(\mathbb{D}) $ and 
$$\sup_{z\in\mathbb{D}}\|f_{z}\|_{ A^{\Psi}_{\omega}(\mathbb{D}) } \lesssim 1.$$ 
\end{lem}

The following lemmas establish some useful properties of these classes of weights.

\begin{lem}(\cite[Lemma 1]{YZ})\label{le: doubling weight property}
Let $\omega$ be a radial weight. Then the following statements are equivalent:

{\rm (i)}  $\omega \in \widehat{\mathcal{D}}$;

{\rm (ii)}  There exist constants  $C=C(\omega)>0$ and $\beta=\beta(\omega)>0$  such that
$$\widehat{\omega}(r)\leq C\left(\frac{1-r}{1-t}\right)^{\beta}\widehat{\omega}(t),\qquad 0\leq r\leq t<1.$$

Moreover, either of the above conditions yields
\begin{align}\label{eq:carleson-square-weight}
  \omega(S(z)) & \thickapprox\widehat{\omega}(z)(1-|z|),\qquad  z\in \mathbb{D},
\end{align}
where the constants depend only on $\omega$.
\end{lem}

\begin{lem}(\cite[Lemma B]{Pd})\label{le:reverse doubling weight property}
Let $\omega$ be a radial weight. Then $\omega\in\check{\mathcal{D}}$ if and only if there exist $C=C(\omega)>0$ and $\alpha=\alpha(\omega)>0$ such that
$$\widehat{\omega}(t)\leq C\left(\frac{1-t}{1-r}\right)^{\alpha}\widehat{\omega}(r),\qquad 0\leq r\leq t<1.$$
\end{lem}

For a radial weight $\omega$. Define
$$\widetilde{\omega}(z)=\frac{\widehat{\omega}(z)}{1-|z|}, \qquad z\in \mathbb{D}.$$

\begin{lem}\label{le:weight_norm_equiv}(\cite[Lemma 9]{DDGW})
Let $\Psi$ be an Orlicz function satisfying the $\Delta_2$-condition and $\omega\in\mathcal{D}$. Then there exists constants $C_{1}=C_{1}(\omega)>0$ and $C_{2}=C_{2}(\omega)>0$ such that for every $f \in H(\mathbb{D})$,
$$C_{1}\int_{\mathbb{D}} \Psi\bigl(|f(z)|\bigr)\omega(z)\,{\rm d} A(z)\leq\int_{\mathbb{D}} \Psi\bigl(|f(z)|\bigr)\widetilde{\omega}(z)\,{\rm d} A(z)\leq C_{2}\int_{\mathbb{D}} \Psi\bigl(|f(z)|\bigr)\omega(z)\,{\rm d} A(z).$$
Consequently, the corresponding Luxemburg norms are equivalent:
$$\|f\|_{A^{\Psi}_{\omega}(\mathbb{D})}\thickapprox\|f\|_{A^{\Psi}_{\widetilde{\omega}}(\mathbb{D})},\qquad f\in H(\mathbb{D}).$$
\end{lem}

\begin{lem}(\cite[Lemma 10]{DDGW})\label{le: pointwise estimate}
Let $\Psi$ be an Orlicz function satisfying the $\Delta_{2}$-condition and $\omega\in \mathcal{D}$. Then there exists a constant $C=C(\omega,\Psi)>0$ such that
$$|f(z)|\leq C\Psi ^{-1}\left(\frac {1} { \omega (S(z))} \right)\|f\|_{A^{\Psi}_{\omega}(\mathbb{D})}$$
for any $f \in A^{\Psi}_{\omega}(\mathbb{D})$ and $z\in \mathbb{D}$.
\end{lem}

The following lemma provides a estimate for the difference of analytic functions in terms of the pseudo-hyperbolic distance, which will be useful in studying difference of composition operators.
\begin{lem}\label{le: difference estimate}
Let $\Psi$ be an Orlicz function satisfying the $\Delta_{2}$-condition and $\omega\in \mathcal{D}$. Let $0 <r_{1}<r_{2}< 1$. Then there exists a constant $C=C(\omega,\Psi,r_{1},r_{2})>0$ such that
 $$| f (a) - f ( z ) |\leq C\Psi^{-1}\left(\frac {1 } {\omega(S(a))}\right)\rho ( a , z )\|f\|_{A^{\Psi}_{\omega}(\mathbb{D})}$$
for any $f \in A^{\Psi}_{\omega}(\mathbb{D})$, $a\in \mathbb{D}$ and $z\in E _{r_{1}} (a)$.  
\end{lem}  
\begin{proof}Fix $z\in E _ { r_{1}} (a)$, $0 <r_{1}<r_{2}< 1$. By \cite[Lemma 3.5]{GW}, we have
$$| f (a) - f (z) |\leq  C_{1}\frac { \rho ( a , z ) } {\omega(S(a))} \int _ { E _ { r_{2}} (a) }| f (\xi)|  \widetilde{\omega}(\xi){\rm d} A(\xi) $$
for some constant $C_{1}=C_{1}(\omega,r_{1},r_{2})>0$. Since $\xi\in E _ { r_{2}} (a)$ and $\omega\in \mathcal{D}$, it follows from \eqref{eq: kernel-estimate}, Lemma \ref{le: doubling weight property} and Lemma \ref{le:reverse doubling weight property}  that $1-|\xi|\thickapprox1-|a|$ and $\widehat{\omega}(\xi)\thickapprox\widehat{\omega}(a)$.
Moreover, by \eqref{eq: area-estimate} and \eqref{eq:carleson-square-weight}, we have
\begin{align*}
  \int _ { E _ { r_{2}} (a) } \widetilde{\omega}(\xi){\rm d} A(\xi)  & =   \int _ { E _ { r_{2}} (a) } \frac{\widehat{\omega}(\xi)}{1-|\xi|}{\rm d} A(\xi) \\
   & \thickapprox \int _ { E _ { r_{2}} (a) } \frac{\widehat{\omega}(a)}{1-|a|}{\rm d} A(\xi)\\
   & \thickapprox  \frac{\widehat{\omega}(a)}{1-|a|} (1-|a|)^{2}\\
   & = \widehat{\omega}(a) (1-|a|)\\
   & \thickapprox \omega(S(a)).
\end{align*}

 Hence, by the non-decreasing property of $\Psi$, Jensen's inequality, and the $\Delta_{2}$-condition, we obtain
\begin{align}\label{le: difference founction estimate} 
\Psi\left(\frac{| f (a) - f ( z ) |}{\rho ( a , z ) }\right) & \leq \Psi\bigg( \frac {C_{1} } {\omega(S(a))} \int _ { E _ { r_{2}} (a) }| f (\xi)|  \widetilde{\omega}(\xi){\rm d} A(\xi)  \bigg)  \nonumber\\
                          & \leq C_{2}\frac { 1 } {\omega(S(a))} \int _ { E _ { r_{2}} (a) }\Psi(| f (\xi)| ) \widetilde{\omega}(\xi){\rm d} A(\xi)
\end{align}
where the constat $C_{2}$ depends only on $\omega,r_{1},r_{2},\Psi$. 
Let $f \in A^{\Psi}_{\omega}(\mathbb{D})$ with $\|f\|_{A^{\Psi}_{\omega}(\mathbb{D})}\neq0$, replacing $f$ in the above inequality by $\frac{f}{\|f\|_{A^{\Psi}_{\omega}(\mathbb{D})}}$ and applying Lemma \ref{le:weight_norm_equiv}, we obtain
\begin{align*}
  \Psi\left(\frac{| f (a) - f ( z ) |}{\rho ( a , z )\|f\|_{A^{\Psi}_{\omega}(\mathbb{D})} }\right) & \leq C_{2}\frac {1 } {\omega(S(a))} \int _ { E _ { r_{2}} (a) } \Psi\left(\frac{| f|}{\|f\|_{A^{\Psi}_{\omega}(\mathbb{D})}} \right)\widetilde{\omega}{\rm d} A \\
   & \leq C_{2}\frac {1 } {\omega(S(a))} \int _ {\mathbb{D}} \Psi\left(\frac{| f|}{\|f\|_{A^{\Psi}_{\omega}(\mathbb{D})}} \right)\widetilde{\omega}{\rm d} A \\
   & \lesssim\frac {1 } {\omega(S(a))}.
\end{align*}
Therefore, 
$$| f (a) - f ( z ) |\leq C\Psi^{-1}\left(\frac {1 } {\omega(S(a))}\right)\rho ( a , z )\|f\|_{A^{\Psi}_{\omega}(\mathbb{D})}$$
for some constat $C=C(\omega,r_{1},r_{2},\Psi)>0$.
\end{proof}

The following lemma characterization of Carleson measures and vanishing Carleson measures for $A^{\Psi}_{\omega}(\mathbb{D})$ follows from \cite[Theorems 6 and 7]{DDGW}, \cite[(2.1)]{DDGW} and the finite  covering property of pseudohyperbolic disk.

\begin{lem}\label{le:Carleson measure}
Let $\Psi$ be an Orlicz function satisfying the $\Delta_{2}$-condition and $\omega\in \mathcal{D}$, $\mu$ be a finite positive  Borel measure on $\mathbb{D}$. Then there exists $r_{0}=r_{0}(\omega)\in(0,1)$ such that for each fixed $r\in (r_{0},1)$, the following statements hold:

{\rm (i)}  $\mu$ is a Carleson measure for $A^{\Psi}_{\omega}(\mathbb{D})$ if and only if $$\sup_{z\in\mathbb{D}}\frac{\mu(E_{r}(z))}{\omega(E_{r}(z))}<\infty.$$

{\rm (ii)} $\mu$ is a vanishing Carleson measure for $A^{\Psi}_{\omega}(\mathbb{D})$ if and only if $$\lim_{|z|\to1}\frac{\mu(E_{r}(z))}{\omega(E_{r}(z))}=0.$$

\end{lem}

For each $\omega\in \mathcal{D}$, there exists $r_{0}=r_{0}(\omega)\in(0,1)$ such that $$ \omega(E_{r}(z))\thickapprox \widehat{\omega}(z)(1-|z|)\thickapprox\omega(S(z)), \,\,z\in\mathbb{D}$$ for each fixed $r\in (r_{0},1)$. See \cite[(2.1)]{DDGW}. 
 
\begin{remark}
By \cite[Lemma 2.4]{GW} and Lemma \ref{le:Carleson measure}, for any $0<p<\infty$, $\mu$ is a (vanishing) Carleson measure for $A^{p}(\omega)$ if and only if it is a (vanishing) Carleson measure for $A^{\Psi}_{\omega}(\mathbb{D})$. 
\end{remark}

Given $\varphi\in S(\mathbb{D})$, $\mu_{\varphi,\omega}$ denote the pull-back measure $$\mu_{\varphi,\omega}(E)=\int_{\varphi^{-1}(E)} \omega(z) {\rm d}A (z)$$
for every Borel subset $E\subset\mathbb{D}$.

The connection between composition operators and Carleson measures comes from the standard identity
\begin{align}\label{eq:change-of-variable}
\int _ { \mathbb{D} } \Psi( | f \circ \varphi (z) |) {\rm d} \mu (z) = \int _ { \mathbb{D} } \Psi(| f (z) |)  {\rm d}( \mu \circ \varphi ^{- 1})  (z)
\end{align}
for any $f\in H(\mathbb{D})$, where $ (\mu \circ \varphi ^{- 1})(E)=\mu(\varphi ^{- 1}(E))$. \\

We recall the following useful property of doubling weights.
\begin{lem}(\cite[Lemma 2]{SQL})\label{le:modified doubling weight}
Let $\omega\in \mathcal{D}$ and $0<\delta<\alpha$, where $\alpha=\alpha(\omega)>0$  is the constant appearing in Lemma \ref{le:reverse doubling weight property}. Let $$\omega_{\delta}(z)=\frac{\omega(z)}{(1-|z|)^{\delta}}, \qquad z\in\mathbb{D}, $$ then there exist constants $C=C(\omega,\delta)>0$ such that 
$$\frac{\widehat{\omega}(z)}{(1-|z|)^\delta}\leq\widehat{\omega_\delta}(z)\leq C\frac{\widehat{\omega}(z)}{(1-|z|)^\delta}$$ and  $\omega_{\delta}\in \mathcal{D}$.
\end{lem}

\section{ Proof of Theorem 1.1}

In this section, we will prove Theorem 1.1. Before proving the main theorem, we need the following auxiliary results.

\begin{lem}\label{le: Orlicz-Carleson}
Let $\Psi$ be an Orlicz function satisfying the $\Delta_{2}$-condition, $\omega\in \mathcal{D}$ and $\varphi \in S(\mathbb{D})$. Set $\delta=\min\{\frac{\alpha}{2},1\}$, where $\alpha=\alpha(\omega)>0$  is the constant appearing in Lemma \ref{le:reverse doubling weight property}. 
For each fixed $r\in (\max\{r_{0}(\omega),r_{0}(\omega_{\delta})\},1)$, there exists a constant $C=C(\omega,r,\Psi,\delta)>0$ such that for every $\varepsilon> 0$ and every Borel function $H: \mathbb{D} \to [0, 1]$ satisfying
$$\sup_ { z \in \mathbb { D } } \left[ H (z) R _ { \varphi } (z) \right] \leq\varepsilon ,$$
we have
$$\int _ { \mathbb { D } } \Psi(| f \circ \varphi |)  H  \omega{\rm d} A  \leq C \varepsilon^{\delta} \int _ { \mathbb { D } } \Psi(| f |)  \omega{\rm d} A  $$
 for any $f\in A^{\Psi}_{\omega}(\mathbb{D})$.
\end{lem}
\begin{proof}
For any Borel set  $E\subset\mathbb { D } $, let $\mu(E)=\int_{\varphi^{-1}(E)} H \omega{\rm d} A$. By \eqref{eq:change-of-variable}, we have$$\int _ { \mathbb { D }} \Psi(| f \circ \varphi |)  H \omega {\rm d} A  =\int _ { \mathbb { D }} \Psi(|f|)  H \omega {\rm d} A \circ \varphi ^{-1}=\int _ { \mathbb { D }} \Psi(|f|) {\rm d} \mu.$$
 Fix $r\in (\max\{r_{0}(\omega),r_{0}(\omega_{\delta})\},1)$. Since $0\leq H\leq1$ and $\delta\leq 1$, we have $H^{1-\delta}\leq 1$. So
\begin{align*}
  \mu(E_{r}(\xi)) & = \int_{\varphi^{-1}(E_{r}(\xi))} H(z) \omega(z){\rm d} A(z)\\
   & \leq \varepsilon^{\delta}\int_{\varphi^{-1}(E_{r}(\xi))} \bigg(\frac{1-|\varphi(z)|^{2}}{1-|z|^{2}}\bigg)^{\delta}\omega(z) {\rm d} A(z).
\end{align*}
For $z\in \varphi^{-1}(E_{r}(\xi))$, it follows from \eqref{eq: kernel-estimate} that $1-|\varphi(z)|^{2}\approx1-|\xi|^{2}$. 
Let $$\omega_{\delta}(z)=\frac{\omega(z)}{(1-|z|)^{\delta}}.$$ 
It follows that
\begin{align*}
  \mu(E_{r}(\xi)) & \lesssim \varepsilon^{\delta} (1-|\xi|^{2})^{\delta} \int_{\varphi^{-1}(E_{r}(\xi))}\frac{\omega(z)}{(1-|z|^{2})^{\delta}}  {\rm d} A(z)\\
   & \thickapprox\varepsilon^{\delta} (1-|\xi|^{2})^{\delta}\int_{\varphi^{-1}(E_{r}(\xi))} \omega_{\delta}(z) {\rm d} A(z)\\
   & = \varepsilon^{\delta} (1-|\xi|^{2})^{\delta} \mu_{\varphi, \omega_{\delta}}(E_{r}(\xi))
\end{align*}
where the constat depend only on $r$ and $\delta$.
Since $0<\delta<\alpha$. By Lemma \ref{le:modified doubling weight}, we get $\omega_{\delta}\in \mathcal{D}$.
Therefore, $C_{\varphi}$ is bounded on $A^{p}(\omega_\delta)$ and $\mu_{\varphi, \omega_{\delta}}$  is a Carleson measure for $A^{p}(\omega_\delta)$. By Remark 2, $\mu_{\varphi, \omega_{\delta}}$  is also a Carleson measure for $A_{\omega_\delta}^{\Psi}(\mathbb D)$. Since $r>r_{0}(\omega_{\delta})$, Lemma \ref{le:Carleson measure} and \eqref{eq: kernel-estimate} yield
\begin{align*}
  \mu(E_{r}(\xi)) & \lesssim \varepsilon^{\delta} (1-|\xi|^{2})^{\delta}\omega_{\delta}(E_{r}(\xi))\\
   & =\varepsilon^{\delta} (1-|\xi|^{2})^{\delta}\int_{E_{r}(\xi)} \omega_{\delta}(z) {\rm d} A(z)\\
   & = \varepsilon^{\delta} (1-|\xi|^{2})^{\delta} \int_{E_{r}(\xi)} \frac{\omega(z)}{(1-|z|)^{\delta}} {\rm d} A(z)\\
   & \thickapprox \varepsilon^{\delta} {\omega(E_{r}(\xi))}.
\end{align*}
Note that $r>r_{0}(\omega)$. Combining this with the subharmonicity of $\Psi(|f|)$ and Fubini's Theorem, Lemma \ref{le:weight_norm_equiv}, we have
\begin{align*}
  \int_{\mathbb{D}}  \Psi(|f(z)|) {\rm d} \mu(z) & \lesssim  \int_{\mathbb{D}}  \frac{1}{\omega(E_{r}(z))} \int_{E_{r}(z)}   \Psi(|f(\xi)|) \widetilde{\omega}(\xi) {\rm d} A(\xi)  {\rm d} \mu(z)\\
   & =\int_{\mathbb{D}}  \int_{E_{r}(\xi)}  \frac{1}{\omega(E_{r}(z))} {\rm d} \mu(z)  \Psi(|f(\xi)|) \widetilde{\omega}(\xi) {\rm d} A(\xi) \\
   & \approx\int_{\mathbb{D}}  \int_{E_{r}(\xi)}  \frac{1}{\omega(E_{r}(\xi))} {\rm d} \mu(z)  \Psi(|f(\xi)|) \widetilde{\omega}(\xi) {\rm d} A(\xi) \\
   & =\int_{\mathbb{D}}  \frac{\mu(E_{r}(\xi))}{\omega(E_{r}(\xi))} \Psi(|f(\xi)|) \widetilde{\omega}(\xi) {\rm d} A(\xi) \\
   & \lesssim  \varepsilon^{\delta} \int_{\mathbb{D}} \Psi(|f(\xi)|) \omega(\xi) {\rm d} A(\xi).
\end{align*}
Hence, 
 $$\int _ { \mathbb { D } } \Psi(| f \circ \varphi |)  H \omega {\rm d} A  \leq C \varepsilon^{\delta} \int _ { \mathbb { D } } \Psi(| f |) \omega {\rm d} A  ,$$
for some constat $C=C(\omega,r,\Psi)>0$.
\end{proof}

Moreover, we shall need the following key estimate.
\begin{lem}\label{le: symbols-difference}
Let $\Psi$ be an Orlicz function satisfying the $\Delta_{2}$-condition and $\omega\in \mathcal{D}$. Let $\varphi, \psi \in S(\mathbb{D})$ and $K \subset \mathbb{D}$ is a Borel set. If
$$\sup _ { z \in K } M _ { \varphi , \psi } (z) \leq \varepsilon  \,\,\, for\,\, some \,\,\, \varepsilon > 0, $$
then there exists a constant $h(\varepsilon) > 0$, such that $\lim\limits _ { \varepsilon\to 0 } h ( \varepsilon ) = 0$   and
$$\int _ { K }\Psi(| f \circ \varphi - f \circ \psi | )\omega {\rm d} A  \leq h ( \varepsilon) \int _ { \mathbb{D}}\Psi(| f |) \omega{\rm d} A $$
for any $f\in A^{\Psi}_{\omega}(\mathbb{D})$.
\end{lem}
\begin{proof}
Assume that $\sup\limits _ { z \in K } M _ { \varphi , \psi } (z) =\sup\limits  _ { z \in K }(R_{\varphi}(z)+R_{\psi}(z))\rho_{\varphi,\psi}(z)\leq \varepsilon$. For $\tau = \tau(\varepsilon )\in (0, \frac{1}{2})$ to be chosen later, and set
$$\Omega _ { \tau } = \left\{ z \in K : \rho_{\varphi,\psi}(z) < \tau \right\},  \,\,\,\, \Omega _ { \tau} ^{\prime}  = K \setminus \Omega _ { \tau } .$$
For any $f\in A^{\Psi}_{\omega}(\mathbb{D})$, write
\begin{align*}
  \int _ { K } \Psi(| f \circ \varphi(z) - f \circ \psi(z) | ) \omega(z){\rm d} A(z)  & =I_{1}(f)+I_{2}(f),
\end{align*}
where $$I_{1}(f)=\int _ { \Omega _ { \tau } }\Psi(| f \circ \varphi(z) - f \circ \psi(z) | ) \omega(z){\rm d} A(z),  \quad I_{2}(f)=\int _ { \Omega _ { \tau} ^{\prime} }\Psi(| f \circ \varphi(z) - f \circ \psi (z)| )\omega(z) {\rm d} A(z). $$

We first estimate $I_{2}(f)$. For $z\in\Omega _ { \tau } ^{\prime}$, we have $\tau\chi_{\Omega_{\tau}^{\prime}}(z)\leq\rho_{\varphi,\psi}(z)$.
Hence
$$\chi _ { \Omega _ { \tau } ^{\prime} } (z) ( R_{\varphi}(z)+R_{\psi}(z)) \leq \frac{1}{\tau}( R_{\varphi}(z)+R_{\psi}(z))\rho_{\varphi,\psi}(z)\leq \frac { \varepsilon} { \tau }.$$
Let $H(z)=\chi _ { \Omega _ { \tau } ^{\prime} } (z)$. We have
 $$\sup_{z\in\mathbb{D}}H(z)R_{\varphi}(z)\leq\sup_{z\in\mathbb{D}}H(z)( R_{\varphi}(z)+R_{\psi}(z)) \leq \frac { \varepsilon} { \tau }.$$
Similarly, $$\sup_{z\in\mathbb{D}}H(z)R_{\psi}(z) \leq \frac { \varepsilon} { \tau }.$$
Applying  Lemma \ref{le: Orlicz-Carleson}, the convexity of $\Psi$ and $\Delta_{2}$-condition, we have
\begin{align}\label{eq: II-f}
  I_{2}(f) &=\int _ { \mathbb{D} }\Psi(| f \circ \varphi(z) - f \circ \psi(z) | )\chi _ { \Omega _ { \tau } ^{\prime} } (z)\omega(z){\rm d} A (z)  \nonumber\\
   &\leq \int _ { \mathbb{D} }\Psi(| f \circ \varphi(z) | +| f \circ \psi(z) | ) H (z)\omega(z){\rm d} A (z)  \nonumber\\
   & \lesssim\int _ {  \mathbb{D}}\Psi(| f \circ \varphi (z)| )H (z)\omega(z){\rm d} A(z) +\int _ {  \mathbb{D}} \Psi(| f \circ \psi(z) | )H(z)\omega(z) {\rm d} A(z)  \nonumber\\
   & \lesssim\left( \frac {\varepsilon } { \tau} \right) ^{\delta} \int _ { \mathbb { D } } \Psi(| f(\xi) |) \omega(\xi) {\rm d} A(\xi) .
\end{align}

We next estimate $I_{1}(f)$.
For $z\in\Omega _ { \tau}$, \eqref{le: difference founction estimate} with $r_1 = 1/2$ and $r_2 = 2/3$ gives
\begin{align*}
 I_{1}(f) & \lesssim \int _ { \Omega _ {\tau } }\frac { \rho_{\varphi,\psi}(z) } {\omega(S ( \varphi(z) ))} \int _ { E _ { 2 / 3} ( \varphi(z) ) } \Psi(| f(\xi) | )\widetilde{\omega}(\xi){\rm d} A (\xi) \omega(z){\rm d} A (z)\\
   & < \tau\int _ { \mathbb{D} }\frac { \chi_{ \Omega _ { \tau}}(z) } {\omega(S( \varphi(z) ))} \int _ {\mathbb{D}}\chi_ { E _ { 2 / 3} ( \varphi(z) ) } (\xi)\Psi(| f(\xi) | )\widetilde{\omega}(\xi){\rm d} A (\xi) \omega(z){\rm d} A (z).
\end{align*}
For any $\xi\in E_{2/3}(\varphi(z))$, combining \eqref{eq: kernel-estimate} with Lemmas \ref{le: doubling weight property} and \ref{le:reverse doubling weight property}, we obtain
$$\omega(S( \varphi(z) ))\thickapprox\widehat{\omega}( \varphi(z))(1-| \varphi(z)|)\thickapprox\widehat{\omega}( \xi)(1-| \xi|)\thickapprox\omega(S (\xi ))$$
Since $\omega\in \mathcal{D}$, composition operator $C_{\varphi}$ is bounded on $A^{p}(\omega)$. So, $\mu_{\varphi,\omega}$ is a Carleson measure for $A^{p}(\omega)$ and hence also for $A_\omega^\Psi(\mathbb{D})$ by Remark 2. Choose a fixed $r>\max\{r_{0}(\omega),\frac{2}{3}\}$, applying Lemma \ref{le:Carleson measure}, we have
$$\mu_{\varphi,\omega}( E _ { 2 / 3} (\xi))\leq\mu_{\varphi,\omega}( E _ {r} (\xi))\lesssim \omega(E _ {r} (\xi))\thickapprox\omega(S(\xi)).$$
That is
$$\sup_{\xi\in\mathbb{D}} \frac{\mu_{\varphi,\omega}( E _ { 2 / 3} (\xi))}{\omega(S(\xi))}<\infty.$$
It follows from Fubini's theorem and Lemma \ref{le:weight_norm_equiv} that 
\begin{align}\label{eq: I-f}
   I_{1}(f) & \lesssim \tau\int _ { \mathbb{D} }\frac { \chi_{ \Omega _ { \tau }}(z)  } {\omega(S (\xi ))} \int _ {\mathbb{D}}\chi_ { \varphi^{-1}( E _ { 2 / 3} (\xi))} (z)\Psi(| f(\xi) | )\widetilde{\omega}(\xi){\rm d} A (\xi) \omega(z){\rm d} A (z)   \nonumber\\
   & =\tau\int _ { \mathbb{D} }\frac {\Psi(| f(\xi) | )} {\omega(S ( \xi))} \int _ { \Omega _ { \tau}\cap\varphi^{-1}( E _ { 2 / 3} (\xi))} \omega(z){\rm d} A (z)\widetilde{\omega}(\xi){\rm d} A (\xi)     \nonumber\\
   &\leq\tau\int _ { \mathbb{D} }\Psi(| f(\xi)| )  \frac{\mu_{\varphi,\omega}( E _ { 2 / 3}(\xi))}{\omega(  S(\xi))}\widetilde{\omega}(\xi){\rm d} A (\xi)      \nonumber\\
   &\leq\tau\sup_{\xi\in\mathbb{D}} \frac{\mu_{\varphi,\omega}( E _ { 2 / 3}(\xi))}{\omega( S (\xi))}\int _ { \mathbb{D} }\Psi(| f(\xi) | ) \widetilde{\omega}(\xi){\rm d} A (\xi)      \nonumber\\
   &\thickapprox\tau\sup_{\xi\in\mathbb{D}} \frac{\mu_{\varphi,\omega}( E _ { 2 / 3}(\xi))}{\omega( S (\xi))}\int _ { \mathbb{D} }\Psi(| f(\xi) | ) \omega(\xi){\rm d} A (\xi) \nonumber\\
   &\lesssim\tau\int _ { \mathbb{D} }\Psi(| f(\xi) | ) \omega(\xi){\rm d} A (\xi).
\end{align}
Combining \eqref{eq: II-f} and \eqref{eq: I-f}, we have
\begin{align*}
 \int _ { K } \Psi(| f \circ \varphi(z) - f \circ \psi(z) | )\omega(z) {\rm d} A(z)  &\leq C [( \frac {\varepsilon } { \tau } ) ^{\delta}+\tau]\int _ { \mathbb { D } } \Psi(| f(\xi) |) \omega(\xi) {\rm d} A(\xi) .
\end{align*}
Taking $\tau=\frac{\sqrt{\varepsilon}}{1+2\sqrt{\varepsilon}}$, $h(\varepsilon)=C [( \frac {\varepsilon } { \tau} ) ^{\delta}+\tau]$, we have  $h(\varepsilon)\rightarrow0$ as $\varepsilon\rightarrow0$, and
$$\int _ { K } \Psi(| f \circ \varphi(z) - f \circ \psi(z) | )\omega(z) {\rm d} A(z) \leq h(\varepsilon)\int _ { \mathbb { D } }  \Psi(| f(\xi) |) \omega(\xi) {\rm d} A(\xi) . $$
\end{proof}

To prove the main theorem, we introduce some notations. Let
\begin{align}\label{eq: T-J}
T _ { J }  &= \sum _ { i \in J } a _ { i } T _ { i }
\end{align}
for $J \subset \Lambda_N$. In view of \eqref{eq: operator-T}, the operators  $T_J$ can be rewritten as
\begin{align}\label{eq: operators-T}
 T _ { J } = \sum _ { i , k \in J } c _ { i , k } T _ { i , k } + \left( \sum _ { i \in J } a _ { i } \right) T _ { \ell },
\end{align}
where $\ell\in J$.

Since $\Psi$ is an Orlicz function satisfying the $\Delta_{2}$-condition, there exists a constant $C=C(N)>0$ such that
\begin{align}\label{eq: subadditivity}
  \Psi(\sum_{i=1}^{N} t_{i}) &\leq C \sum_{i=1}^{N}\Psi( t_{i})
\end{align}
for all $t_{i}>0,i\in\Lambda_N$.

\vskip.2cm

\textbf{Proof of Theorem 1.1}.
First, proof of ${\rm (a)}\Rightarrow {\rm (b)}$.
Assume that $T$ is compact on $A^{\Psi}_{\omega}(\mathbb{D}) $. For any $z\in\mathbb{D}$, let 
$$f_{z,\gamma}(\xi)=\bigg(\frac{1-|z|^{2}}{1-\bar{z}\xi}\bigg)^{\gamma+1}\Psi^{-1}\bigg(\frac{1}{\omega(S(z))}\bigg). $$
By Lemma \ref{le: text-function}, for all $\gamma\geq\gamma_{0}$, we have $\|f_{z}\|_{A^{\Psi}_{\omega}(\mathbb{D})}\lesssim1$. Moreover, $f_{z,\gamma}\to0$ uniformly on compact subsets of $\mathbb{D}$ as $|z|\to1$. Hence, Lemma \ref{le:compactness} implies that $$\lim_{|z|\to1}\|Tf_{z,\gamma}\|_{A^{\Psi}_{\omega}(\mathbb{D})}=0.$$

Next, proof of ${\rm (b)}  \Rightarrow {\rm (c)}$.  Assume that ${\rm (b)}$ holds. We suppose that ${\rm (c)}$ fails and will derive a contradiction.

Since ${\rm (c)}$ fails, there exists a sequence $\{z_n\} \subset \mathbb{D}$ such that $|z_n| \to 1$ and
$$\inf _ { n } Q ( z _ { n } ) > 0 .$$
Note that the functions $Q_{\eta}$ are bounded on $\mathbb{D}$ and $Q ( z _ { n } ) =\prod_{\eta \in \mathscr{ P } _ { N }}Q_{\eta}(z_n)$, we obtain
\begin{align}\label{eq: inf-Q}
\inf _ { n } Q _ { \eta } ( z _ { n } ) > 0 ,   \,\,\,\,\,\,\,\,\eta \in \mathscr{ P } _ { N }.
\end{align}
Furthermore, we suppose
$\{ R _ { i } ( z _ { n } ) \} $ and $ \{ \rho _ { i , k } ( z _ { n } ) \} $  both converge as $n \to \infty$, for every $i,k\in\Lambda_{N}$. Let
\begin{align}\label{eq: define-F}
F = \left\{ i \in \Lambda _ { N } :\lim _ { n \to \infty } R _ { i } ( z _ { n } ) = 0 \right\}
\end{align}
and
$$F ^{\prime}  = \Lambda _ { N } \setminus F.$$
Assume that
\begin{align}\label{eq: inf-R}
\inf_ { n } R _ { i } ( z _ { n } ) > 0
\end{align}
for each $i \in F'$.

We interrupt the proof to show that the set $F'$ satisfies the following claim.

{\bf Claim}: There exist pairwise disjoint sets of indices $J_1, \ldots, J_L$ and a subsequence of $\{z_n\}$ satisfies the following properties:

{\rm (i)} $F'=\bigcup\limits_{m=1}^{L}J_{m}$;

{\rm (ii)} For each $m=1,\ldots,L $, there exists $k_m\in J_m $ such that for all $n $ and all $i \in \bigcup\limits _ { \ell = m } ^{L} J _ { \ell } $,
$$| \varphi _ { i } ( z _ { n } ) | \leq | \varphi _ { k _ { m } } ( z _ { n } ) |; $$

{\rm (iii)} For each $m=1, \ldots, L$, $\lim\limits _{n \rightarrow \infty} \rho_{i, k}\left(z_{n}\right)=0, \qquad i, k \in J_{m}$ ;

{\rm (iv)} $\sum\limits_{i \in J_{m}} a_{i}=0$ for all $m=1, \ldots, L$.

Proof of the Claim. The proofs of {\rm (i)}-{\rm (iii)} can be found in \cite{CKW}. It remains only to prove {\rm (iv)}. We first consider the case $m=1$.   Let
$$f_{a}(u)=\bigg(\frac{1-|a|^{2}}{1-\bar{a}u}\bigg)^{\gamma+1}\Psi^{-1}\bigg(\frac{1}{\omega(S(a))}\bigg),  \,\,\,\,\, a\in \mathbb{D}. $$
Set $$\xi _ { n } = \varphi _ { k _ { 1 } } ( z _ { n } )$$
for each $n$.  By \eqref{eq: inf-R}, we have $1 - |\xi_n|^2 \approx 1 - |z_n|^2$. Hence, $|\xi_n| \to 1$. 
 (b) and Lemma \ref{le: text-function} yield
$$\lim _ { n \to \infty } \| T f _ { \xi _ { n } } \| _ { A ^{\Psi}_{\omega}(\mathbb{D}) }=0.$$
Moreover, combining \eqref{eq:carleson-square-weight} with Lemma \ref{le: doubling weight property} and \ref{le:reverse doubling weight property}, we obtain
$$\omega(S(z_{n}))\thickapprox\widehat{\omega}(z_{n})(1-|z_{n}|)\thickapprox\widehat{\omega}(\xi_{n})(1-|\xi_{n}|)\thickapprox\omega(S(\xi_{n})) $$
 Applying Lemma \ref{le: pointwise estimate}, we have
\begin{align*}
| Tf _ { \xi _ { n } } ( z_{n} ) |&\lesssim \Psi^{-1}\left(\frac { 1 } { \omega(S(z_{n})) }\right) \| T f _ { \xi _ { n } } \| _ { A ^{\Psi}_{\omega}(\mathbb{D}) } \\
                                  &\thickapprox\Psi^{-1}\left(\frac { 1 } { \omega(S(\xi_{n})) }\right) \| T f _ { \xi _ { n } } \| _ { A ^{\Psi}_{\omega}(\mathbb{D}) }.
\end{align*}
Hence
\begin{align}\label{eq: lim-T}
\lim _ { n \to \infty } \frac{| Tf _ { \xi _ { n } } ( z_{n} ) |}{\Psi^{-1}\left(\frac { 1 } { \omega(S(\xi_{n})) }\right)}= 0.
\end{align}

For $i \in J _ { 1 }$, Claim {\rm (iii)} implies that $\lim \limits_{n \rightarrow \infty} \rho_{i, k_{1}}\left(z_{n}\right)=0$.
It follows from \eqref{eq: ratio-estimate} that
$$1 - | \xi _ { n } |= 1- |\varphi_{k_{1}}(z_{n}) | \approx1 - | \varphi _ { i } ( z _ { n } ) | .$$
For $i,k \in J _ { 1 }$. By Lemma \ref{le: difference estimate}, \ref{le: doubling weight property} and \ref{le:reverse doubling weight property}, we have
\begin{align*}
|T _ { i,k  }f _ { \xi _ { n } } ( z _ { n } ) |
     & \lesssim \Psi^{-1}\left(\frac{1}{\omega(S ( \varphi_ {i }( z _ { n }) ))} \right)\rho_{i,k} (z _ { n })\|  f _ { \xi _ { n } } \| _ { A ^{\Psi}_{\omega}(\mathbb{D}) } \\
    & \thickapprox\Psi^{-1}\left(\frac{1}{\omega(S ( \xi _ { n } ))} \right)\rho_{i,k} (z _ { n })\|  f _ { \xi _ { n } } \| _ { A ^{\Psi}_{\omega}(\mathbb{D}) }.
\end{align*}
Since $\|  f _ { \xi _ { n } } \| _ { A ^{\Psi}_{\omega}(\mathbb{D}) } \lesssim 1$, Claim {\rm (iii)} yields
\begin{align}\label{eq: T-ik}
 \frac{|T _ { i,k  }f _ { \xi _ { n } } ( z _ { n } ) |}{\Psi^{-1}\left(\frac{1}{\omega(S ( \xi _ { n } ))} \right)} & \lesssim \rho_{i,k} (z _ { n })\to 0.
\end{align}

For $i\in \Lambda_{N}$, we get
\begin{align}\label{eq: T-i}
    \frac{| T _ { i } f _ { \xi _ { n } } ( z _ { n } ) |}{\Psi^{-1}\left(\frac{1}{\omega(S(\xi _ { n }))}\right)}
   & =\frac{1}{\Psi^{-1}\left(\frac{1}{\omega(S(\xi _ { n }))}\right)} \bigg(\frac{1-| \xi _ { n } | ^{2}}{ | 1 - \overline{\xi _ { n }}\varphi_ {i }( z _ { n }) |}\bigg)^{\gamma+1}\Psi^{-1}\left(\frac{1}{ \omega(S(\xi _ { n }))}\right)   \nonumber\\
   & =\bigg(\frac{1-| \xi _ { n } | ^{2}}{ | 1 - \overline{\xi _ { n }}\varphi_ {i }( z _ { n }) |}\bigg)^{\gamma+1}.
\end{align}
In particular, when $i=k_{1}$, we have
 $$ \frac{| T _ { k_{1} } f _ { \xi _ { n } } ( z _ { n } ) |}{\Psi^{-1}\left(\frac{1}{\omega(S(\xi _ { n }))}\right)}  =1.$$ 
Taking $J=J_{1}$ and $\ell=k_{1}$ in \eqref{eq: operators-T}. Combining \eqref{eq: T-ik}, we obtain
\begin{align}\label{eq: T-J1}
   \lim_{n\to\infty}\frac{|T _ { J_{1} }  f _ { \xi _ { n } } ( z _ { n } )|}{\Psi^{-1}\left(\frac{1}{\omega(S(\xi _ { n }))}\right)} & =\big| \sum _ { i \in J_{1} } a _ { i } \big|.
\end{align}

For $i\in F$, observe that $| 1 -\overline { { \xi _ { n } } } \varphi _ { i } ( z _ { n } ) |\geq 1-|\varphi _ { i } ( z _ { n } )|\thickapprox1 -  |\varphi _ { i } ( z _ { n } ) |^{2}$.
Thus \eqref{eq: define-F} and \eqref{eq: T-i} imply that
\begin{align*}
\frac{| T _ { i } f _ { \xi _ { n } } ( z _ { n } ) |}{\Psi^{-1}\left(\frac{1}{\omega(S(\xi _ { n }))}\right)}
  & \lesssim\left( \frac { 1 - | \xi _ { n } | ^{2} } { 1 -  |\varphi _ { i } ( z _ { n } ) |^{2}  } \right) ^{\gamma+1} \\
   & \approx\left( \frac { 1 - | z _ { n } | ^{2} } { 1 -  |\varphi _ { i } ( z _ { n } ) |^{2}  } \right) ^{\gamma+1} = R _ {i } ^{\gamma+1} ( z _ { n } ) \rightarrow  0.
\end{align*}

Combining this with \eqref{eq: T-J}, we obtain
\begin{align*}
\frac{| T _ { F } f _ { \xi _ { n } } ( z _ { n } ) |}{\Psi^{-1}\left(\frac{1}{\omega(S(\xi _ { n }))}\right)} & =\frac{|  \sum _ { i \in F } a _ { i } T _ { i } f _ { \xi _ { n } } ( z _ { n } ) |}{\Psi^{-1}\left(\frac{1}{\omega(S(\xi _ { n }))}\right)}\\
     & \leq  \sum _ { i \in F } |a _ { i }|\frac{| T _ { i } f _ { \xi _ { n } } ( z _ { n } ) |}{\Psi^{-1}\left(\frac{1}{\omega(S(\xi _ { n }))}\right)}.
\end{align*}
Hence
\begin{align}\label{eq: T-F}
\lim_{n\to\infty}\frac{| T _ { F } f _ { \xi _ { n } } ( z _ { n } ) |}{\Psi^{-1}\left(\frac{1}{\omega(S(\xi _ { n }))}\right)}= 0 .
\end{align}

Next, we consider the operators $T_{J_k}$ for $1 < k \leq L$. By Claim (ii), $|\varphi_i(z_n)| \leq |\xi_n|, \,\,i \in J_k$. It follows from \eqref{eq: distance} that
\begin{align*}
  \left( \frac { 1 - | \xi _ { n } | ^{2} } { | 1 - \varphi _ { i } ( z _ { n } ) \overline { { \xi _ { n } } } | } \right) ^{2} & \leq \frac { ( 1 - | \xi _ { n } | ^{2} ) ( 1 - | \varphi _ { i } ( z _ { n } ) | ^{2} ) } { | 1 - \varphi _ { i } ( z _ { n } ) \overline { { \xi _ { n } } } | ^{2} }  \\
   & =\frac { ( 1 - | \varphi _ { k_{1} } ( z _ { n } ) | ^{2} ) ( 1 - | \varphi _ { i } ( z _ { n } ) | ^{2} ) } { | 1 - \varphi _ { i } ( z _ { n } ) \overline { {\varphi _ { k_{1}} ( z _ { n } ) } } | ^{2} }= 1 - \rho _ { k_{1} , i } ^{2} ( z _ { n } ).
\end{align*}
 From \eqref{eq: T-i}, we obtain
\begin{align*}
     \frac{| T _ { i } f _ { \xi _ { n } } ( z _ { n } ) |}{\Psi^{-1}\left(\frac{1}{\omega(S(\xi _ { n }))}\right)}
   & \leq ( 1 - \rho _ {k_{1} , i } ^{2} ( z _ { n } ) )^{\frac{\gamma+1}{2}}
\end{align*}
for $i \in J_k$.

Let $$\eta _ { 1 }  = \min\left[ \lim _ { n \to \infty } \rho _ { k _ { 1 } , i } ( z _ { n } ) \right] > 0,$$
where the minimum is taken over all $i \in \cup_{1 < k \leq L} J_k$.
 By \eqref{eq: T-J}, we have
\begin{align}\label{eq: T-Jk}
  \limsup_{n\rightarrow\infty}\sum_{1 < k \leq L} \frac{| T _ { J _ { k } } f _ {\xi_ { n } } ( z _ { n } ) | }{\Psi^{-1}\left(\frac{1}{\omega(S(\xi _ { n }))}\right)} & \leq \limsup_{n\rightarrow\infty} \sum_{1 < k \leq L}\sum_{i\in J_k} |a_{i}|\frac{| T _ { i } f _ { \xi _ { n } } ( z _ { n } ) |}{\Psi^{-1}\left(\frac{1}{\omega(S(\xi _ { n }))}\right)} \nonumber\\
   &  \leq \limsup_{n\rightarrow\infty} \sum_{1 < k \leq L}\sum_{i\in J_k} |a_{i}|( 1 - \rho _ {k_{1} , i } ^{2} ( z _ { n } ) )^{\frac{\gamma+1}{2}} \nonumber\\
   & \leq N \sum_{i=1}^{N} |a_{i}|( 1 -  \eta _ { 1 }^{2} ( z _ { n } ) )^{\frac{\gamma+1}{2}} \nonumber\\
    & =C( 1 -  \eta _ { 1 }^{2} ( z _ { n } ) )^{\frac{\gamma+1}{2}} 
\end{align}
 for some constat $C=C(N,a_{1},\cdot\cdot\cdot,a_{N})>0$.

Applying \eqref{eq: T-J} and Claim {\rm (i)}, we obtain  $$T _ { J _ { 1 } } = T - T _ { F } - \sum _ { 1 < k \leq L } T _ { J _ { k } } .$$
Moreover,
\begin{align*}
  \frac{| T _ { J _ { 1 } } f _ { \xi _ { n } } ( z _ { n } ) | }{\Psi^{-1}\left(\frac{1}{\omega(S(\xi _ { n }))}\right)}& =\frac{|  ( T - T _ { F } - \sum _ { 1 < k \leq L } T _ { J _ { k } } ) f _ { \xi _ { n } } ( z _ { n } )|}{\Psi^{-1}\left(\frac{1}{\omega(S(\xi _ { n }))}\right)} \\
   & \leq  \frac{| T f _ { \xi _ { n } } ( z _ { n } )|}{\Psi^{-1}\left(\frac{1}{\omega(S(\xi _ { n }))}\right)}+ \frac{| T _ { F }f _ { \xi _ { n } } ( z _ { n } )|}{\Psi^{-1}\left(\frac{1}{\omega(S(\xi _ { n }))}\right)}+ \sum _ { 1 < k \leq L }\frac{| T _ { J _ { k } }f _ { \xi _ { n } } ( z _ { n } )|}{\Psi^{-1}\left(\frac{1}{\omega(S(\xi _ { n }))}\right)}.
\end{align*}

Combinin \eqref{eq: lim-T}, \eqref{eq: T-J1}, \eqref{eq: T-F}, and \eqref{eq: T-Jk}, we conclude that
$$\big| \sum _ { i \in J_{1} } a _ { i } \big|\leq C( 1 - \eta _ { 1 } ^{2} )^{\frac{\gamma+1}{2}} .$$
Since $\gamma>0$ is arbitrary. Taking the limit $\gamma\to \infty$, we arrive at
 $$\sum _ { i \in J _ { 1 } } a _ {i } = 0 .$$
Hence,  the proof is complete for the case $m=1$.

Now let $2 \leq m \leq L$ and assume that
$$\sum _ { i \in J _ { 1 } } a _ { i } = \cdots = \sum _ { i \in J _ { m - 1 } } a _ { i } = 0 .$$
It remains to show that
$$\sum _ { i \in J _ { m } } a _ { i } = 0. $$
Since the proof is  analogous to the case $m=1$, we omit the details. This completes the proof of the claim.

Based on the claim, we now return to the proof of ${\rm (b)}  \Rightarrow {\rm (c)}$. One may assume that the sets $J_1,\ldots,J_L,F$ are all nonempty. Let $n_m$ and $d$ denote the number of elements in $J_m$ and $F$, respectively. Let
 $$\mathfrak  { j } _ { m }  = ( i _ { m , 1 } , \ldots , i _ { m , n _ { m } } ) ,  \qquad m = 1 , \ldots , L,$$
where $i_{m,1},\ldots,i_{m,n_m}$ are the distinct elements of $J_m$. Let $$\mathfrak { f }  = ( i _ { 1 } , \ldots , i _ { d } ),$$
where $i_{1},\ldots,i_{d}$ are the distinct elements of $F$. Constructing  $$\sigma = ( \mathfrak  { j } _ { 1 } , \ldots , \mathfrak { j } _ { L } , \mathfrak { f } )\in \mathscr{P}_N,$$
by \eqref{eq: define-Q}, we conclude that
 \begin{align*}
    Q _ { \sigma }  & = \sum _ { i = 1 } ^{N - 1} | s _ { i } ^{\sigma} | M _ { \sigma _ {i } ,\sigma _ { i + 1 } } + | s _ { N } ^{\sigma} | R _ { \sigma _ { N } } \\
    & = \bigg(\sum _ { i = 1 } ^{N - d} | s _ { i } ^{\sigma} | M _ { \sigma _ { i } , \sigma _ { i + 1 } }+\sum _ { i = N -d+1 } ^{N - 1} | s _ { i } ^{\sigma} | M _ { \sigma _ { i } , \sigma _ { i + 1 } }\bigg) + | s _ { N } ^{\sigma} | R _ { \sigma _ { N } } \\
    & =\sum _ { i = 1 } ^{N - d} | s _ { i } ^{\sigma} | M _ { \sigma _ { i } , \sigma _ { i + 1 } }+\bigg(\sum _ { i = N -d+1 } ^{N - 1} | s _ { i } ^{\sigma} | M _ { \sigma _ { i } , \sigma_ { i + 1 } } + | s _ { N } ^{\sigma} | R _ { \sigma _ { N } }\bigg).
 \end{align*}
 Applying Claim (iv) with $\sigma$, we deduce that
$$\sum _ { i = 1 } ^{n _ { 1} + \cdots + n _ { m } } a _ { \sigma _ { i } } = 0   \,\,\,\,\,\,\,   m=1,\ldots,L.$$
So
$$\sum _ { i = 1 } ^{N - d} \left| s _ { i } ^{\sigma} \right| M _ { \sigma _ { i } , \sigma _ { i + 1 } } = \sum _ { m = 1 } ^{L} \sum _ { i , k \in J _ { m } } c _ { i , k } ^{( m )} M _ { i , k },$$
where $c_{i,k}^{(m)}$ are nonnegative coefficients depending on $\{a_{\ell}:\ell\in J_{m}\}$.

On the other hand, noting that $M _ { \sigma _ { i } , \sigma _ { i + 1 } }=(R _ {\sigma_ { i } }+R _ { \sigma _ { i +1} })\rho_{\sigma_ { i } , \sigma _ { i + 1 }}$ and $0\leq\rho _ { \sigma _ { i } , \sigma _ { i + 1 } }<1$, so
 $$\sum _ { i = N - d + 1 } ^{N - 1} \left| s _ { i } ^{\sigma} \right| M _ {\sigma _ { i } , \sigma _ { i + 1 } } + \left| s _ { N } ^{\sigma} \right| R _ { \sigma _ { N } } \leq \sum _ { i \in F } c _ { i } R _ { i },$$
where $c_i$ are nonnegative coefficients depending on $\{a_{i_1},\ldots,a_{i_d}\}$.
Hence
$$Q _ { \sigma} (z_n)\leq\sum _ { m = 1 } ^{L} \sum _ { i , k \in J _ { m } } c _ { i , k } ^{( m )} M _ { i , k }(z_n) + \sum _ { i \in F } c _ { i } R _ { i }(z_n).$$
By Claim {\rm (iii)}, for $i,k\in J_{m}$, we have $\lim\limits _{n \rightarrow \infty} \rho_{i, k}(z_{n})=0$. From $M _ { i, k } ( z_{n} )  = [ R _ { i } ( z_{n} ) + R _ { k } ( z_{n} ) ] \rho _ { i , k } ( z_{n}) $, we obtain $\sum\limits _ { m = 1 } ^{L} \sum\limits _ { i , k \in J _ { m } } c _ { i , k } ^{( m )} M _ { i , k }(z_n)\rightarrow 0$.
Meanwhile, by \eqref{eq: define-F}, we conclude that $ \sum\limits _ { i \in F } c _ { i } R _ { i }(z_n)\rightarrow 0$.
So $$\lim _ { n \to \infty } Q _ { \sigma } ( z _ { n } ) = 0 .$$
This contradicts \eqref{eq: inf-Q}. Therefore  $\lim\limits_{|z|\to1}Q(z)=0$. This completes the proof.

Finally, we prove ${\rm (c)} \Rightarrow {\rm (a)}$.
Let $\{f_n\}$ be an arbitrary sequence in $A^{\Psi}_{\omega}(\mathbb{D})$ such that $\sup_{n} \|f_n\|_{A^{\Psi}_{\omega}(\mathbb{D})} \leq 1$ and $f_n\rightarrow0$ uniformly on compact subsets of  $\mathbb{D}$. By Lemma  \ref{le:compactness}, it suffices to show that $\|Tf_n\|_{A^{\Psi}_{\omega}(\mathbb{D})}\rightarrow 0$. Moreover, by \cite[Theorem 9.4]{KR}, we need to prove that $\lim \limits_ { n \to \infty } \int _ { \mathbb{D} }\Psi ( | Tf _ { n }  | ) \omega{\rm d}A = 0 $.

Let $\varepsilon>0$ and $\eta\in\mathscr{P}_N$. Define
$$U_{\eta,\varepsilon}=\left\{z\in\mathbb{D}:Q_{\eta}(z)\leq\varepsilon\right\}.$$
For $\xi\in \mathbb{T}$ and $0<h<1$, let $I(\xi,h)$ denote the open arc of $\mathbb{T}$ centered at $\xi$ with arc length $2\pi h$. The Carleson square
associated with $I(\xi,h)$ is defined by
$$S(\xi,h)=\left\{z\in\mathbb{D}:1-h<|z|<1, \frac{z}{|z|}\in I(\xi,h)\right\}.$$
By the assumption (c), for every $\xi\in \mathbb{T}$, there exists $h_{\xi}=h_{\xi}(\varepsilon)\in(0,1)$
such that
\begin{align}\label{eq:U-cover}
  S(\xi,h_{\xi}) & \subset\bigcup_{\eta\in\mathscr{P}_N}U_{\eta,\varepsilon}.
\end{align}
Otherwise, there exist some $\xi_0\in \mathbb{T}$ such that, for every $h\in(0,1)$, the set $S(\xi_0,h)$ contains a point
$z_{h}\in\mathbb{D}$ satisfying $Q_{\eta}(z_{h})>\varepsilon$ for all $\eta\in\mathscr{P}_N$. Hence, we have 
$Q(z_{h})>\varepsilon^{N!}$, which contradicts (c).

Since $\mathbb{T}$ is compact, there exist finitely many points $\xi_1,\xi_2,\ldots,\xi_{j}\in \mathbb{T}$
such that  $\mathbb{T}=\bigcup_{i=1}^{j}I(\xi_i,h_i)$, where $h_i=h_{\xi_i}$.
 Set $r=\max_{1\leq i\leq j}(1-h_i)$ and  $r \mathbb { D }=\{z\in \mathbb {C }: 0<|z|< r\}$. Then each $ z\in\mathbb { D } \backslash r \mathbb { D }$ belongs to $S(\xi_i,h_i) $ for some $1\leq i\leq j$. Hence
\begin{align*}
  \mathbb{D}\setminus r\mathbb{D} & \subset\bigcup_{i=1}^{j}S(\xi_i,h_i).
\end{align*}
Combining  this with \eqref{eq:U-cover} yield
$$\mathbb{D}\setminus r\mathbb{D}\subset\bigcup_{i=1}^{j}S(\xi_i,h_i)\subset\bigcup_{\eta\in\mathscr{P}_N}U_{\eta,\varepsilon}.$$
Therefore
\begin{align*}
    \int _ { \mathbb{D} }\Psi ( | Tf _ { n }| )   \omega {\rm d}A &=  \int _ { r\mathbb { D } } \Psi ( | Tf _ { n }  | ) \omega   {\rm d}A+ \int _ { \mathbb { D } \backslash r\mathbb { D }}\Psi ( | Tf _ { n }  | )  \omega  {\rm d}A  \\
  & \leq \int _ { r\mathbb { D } } \Psi ( | Tf _ { n }  | )  \omega  {\rm d}A+  \sum _ { \eta \in \mathscr { P } _ { N } }\int _ { U _ { \eta , \varepsilon }} \Psi ( | Tf _ { n }  | )   \omega {\rm d}A\\
   & =I _ { 1 } + \sum _ { \eta \in \mathscr { P } _ { N } } I _ { 2 }.
\end{align*}

We first estimate $I _ { 1 }$ and $I _ { 2 }$, respectively.
Since $\bigcup_{i=1}^{N}\varphi_{i}(r\mathbb { D })$  is a compact subset of $\mathbb { D }$,
and $f_{n}\to0$ uniformly on $\bigcup_{i=1}^{N}\varphi_{i}(r\mathbb { D })$, we obtain
\begin{align}\label{eq: I-1}
I_{1}\rightarrow0 \quad  as  \quad  n\rightarrow\infty.
\end{align}
For $I _ { 2 } $,
by \eqref{eq: operator-T}, \eqref{eq: subadditivity} and the non-decreasing of $\Psi$, we arrive at
\begin{align*}
   I_{2} & \leq \int_{U_{\eta,\varepsilon}}\Psi\bigg(\sum_{i=1}^{N-1}|s_{i}^{\eta}||T_{\eta_{i},\eta_{i+1}}f _ { n } |+|s_{N}^{\eta}||T_{\eta_{N}}f _ { n } |\bigg)\omega {\rm d}A \\
   & \lesssim\sum_{i=1}^{N-1}\int_{U_{\eta,\varepsilon}}\Psi\bigg(|s_{i}^{\eta}||T_{\eta_{i},\eta_{i+1}}f _ { n } |\bigg) \omega{\rm d}A+\int_{U_{\eta,\varepsilon}}\Psi\bigg(|s_{N}^{\eta}||T_{\eta_{N}}f _ { n } |\bigg)\omega {\rm d}A.
\end{align*}
Since the terms with $s_{i}^{\eta}=0$ vanish in the decomposition of $Tf_n$, we need only consider those indices $i$ for which $|s_i^\eta|>0$.
 Hence, by \eqref{eq: remark 1}, we have
\begin{align*}
    I_{2}&\lesssim\sum_{i=1}^{N-1}\int_{U_{\eta,\varepsilon}}\Psi\bigg(|T_{\eta_{i},\eta_{i+1}}f _ { n } |\bigg) \omega{\rm d}A+\int_{U_{\eta,\varepsilon}}\Psi\bigg(|T_{\eta_{N}}f _ { n } |\bigg) \omega{\rm d}A\\
   &=\sum_{i=1}^{N-1} J_{1}+J_{2}.
\end{align*}

 We now estimate $J_{1}$ and $J_2$.
For $ z\in U_{\eta,\varepsilon}$, $Q_{\eta}(z)\leq\varepsilon$. It follows from \eqref{eq: define-Q} that $M_{\eta_{i},\eta_{i+1}}(z)\lesssim\varepsilon$.
Let $K=U_{\eta,\varepsilon}$, we get $$\sup_{z\in K}M_{\eta_{i},\eta_{i+1}}(z)\lesssim\varepsilon.$$ Applying Lemma \ref{le: symbols-difference}, there exists a constant
 $h_{i}(\varepsilon)>0$ such that $\lim\limits_{\varepsilon\rightarrow0}h_{i}(\varepsilon)=0$ and
\begin{align*}
   \sum_{i=1}^{N-1} J_{1}& \leq  \sum_{i=1}^{N-1} h_{i}(\varepsilon)\int_{\mathbb{D}}\Psi(|f_{n}|)\omega{\rm d}A\leq \sum_{i=1}^{N-1} h_{i}(\varepsilon).
\end{align*}
For $J_{2}$, a similar argument gives $R_{\eta_{N}}(z)\lesssim \varepsilon,\,\,z\in U_{\eta,\varepsilon}$. Set $H(z)=\chi_{U_{\eta,\varepsilon}}(z)$. Then  $$\sup_{z\in\mathbb{D} }H(z)R_{\eta_{N}}(z)\leq\sup_{z\in U_{\eta,\varepsilon} }R_{\eta_{N}}(z)\lesssim\varepsilon.$$ By Lemma \ref{le: Orlicz-Carleson}, there exists constants $C>0$ and $\delta$, independent of $\varepsilon$, such that
\begin{align*}
 J_{2}
   & = \int_{\mathbb{D} }\Psi(|f_{n}\circ\varphi_{\eta_{N}}|) H\omega{\rm d}A \\
   & \leq  C \varepsilon^{\delta}\int_{\mathbb { D }}\Psi(|f_{n}|) \omega{\rm d}A  \leq  C \varepsilon^{\delta}.
\end{align*}
Hence
\begin{align}\label{eq: I-2}
I _ { 2}  \leq\sum_{i=1}^{N-1}  h_{i}  ( \varepsilon ) + C\varepsilon ^{\delta}.
\end{align}
Combining \eqref{eq: I-1} and \eqref{eq: I-2}, we obtain
$$\limsup_{n\rightarrow\infty}\int_\mathbb { D }\Psi( |Tf_n|)\omega {\rm d}A\leq\sum_{i=1}^{N-1}h_{i}  ( \varepsilon ) + C\varepsilon ^{\delta}.$$
The constant suppressed in the estimate is independent of $\varepsilon$. Taking the limit $\varepsilon\rightarrow0$, we have $\lim_{n\rightarrow\infty}\int_\mathbb { D }\Psi( |Tf_n|)\omega {\rm d}A=0$.
Therefore, $T$ is compact on $A^{\Psi}_{\omega}(\mathbb{D})$. This completes the proof of Theorem 1.1.

\section{ Proof of Theorem 1.2}

In what follows, we will devoted to the proof of Theorem 1.2. Let $\eta\in\mathscr{P}_N$. By \eqref{eq:change-of-variable}, the corresponding joint pullback measure satisfies the following identity for any positive Borel function $g$ on $\mathbb{D}$:
\begin{align}\label{eq: measure}
  \int _ { \mathbb { D } } \Psi(|g| ) {\rm d} \mu _ {\Psi,\eta,\omega} & = \sum _ { i = 1 } ^{N-1} | s _ { i } ^{\eta} | \int _ {G _ { \eta } } \left[\Psi(|g\circ\varphi _ { \eta _ { i } } | )  +  \Psi(|g\circ\varphi _ { \eta _ { i+1 } } | )\right]\Psi( M_{\eta _ { i },\eta _ { i+1 }}) \omega {\rm d}A \nonumber\\
   & +| s _ {N } ^{\eta} |\int _ { G _ { \eta } } \Psi(|g\circ\varphi _ { \eta_ { N } }| )\Psi( R _ { \eta _ { N } } ) \omega {\rm d}A.
\end{align}

The following lemma plays a vital role in the proof of Theorem 1.2.
\begin{lem}\label{le: vanishing-Carleson-measure}
Let $\Psi$ be an Orlicz function satisfying the $\Delta_{2}$-condition, $\omega\in \mathcal{D}$ and $\varphi, \psi \in S(\mathbb{D})$. Suppose that $H : \mathbb{D} \to [0,1]$ be a Borel function. Define measures $\mu$ and $\nu$ on $\mathbb{D}$ by
$$\mu (E) =  \int _ { \varphi ^{- 1} (E) }\Psi ( M _ { \varphi, \psi  } )  H \omega{\rm d} A +  \int _ { \psi ^{- 1}(E) } \Psi ( M _ { \varphi, \psi  } )  H \omega {\rm d} A $$
and
$$\nu (E) = \int _ { \varphi ^{- 1} (E)} \Psi ( \rho _ { \varphi, \psi  } ) H \omega {\rm d} A  + \int _ { \psi ^{- 1} (E)} \Psi ( \rho _ { \varphi, \psi  } ) H  \omega{\rm d} A ,$$
for any Borel set $E\subset \mathbb{D}$. If $\mu$ is a vanishing Carleson measure for $A_{\omega}^{\Psi}(\mathbb{D})$, then so is $\nu$.
\end{lem}
\begin{proof}
For $0<\varepsilon<1$, let $K_{\varepsilon}=\{z\in\mathbb{D}: R_{\varphi}(z)+R_{\psi}(z)\leq\varepsilon\}$. For any Borel set $E\subset \mathbb{D}$, put
\begin{align*}
  \int _ { \varphi ^{- 1} (E)} \Psi ( \rho _ { \varphi, \psi  } ) H \omega{\rm d} A  & =\int _ { \varphi ^{- 1} (E)\cap K_{\varepsilon}} \Psi ( \rho _ { \varphi, \psi  } ) H \omega {\rm d} A +\int _ { \varphi ^{- 1} (E)\setminus K_{\varepsilon}} \Psi ( \rho _ { \varphi, \psi  } ) H \omega {\rm d} A  \\
   & =\int _ {\varphi^{- 1} (E) } \chi_{ K_{\varepsilon}}\Psi ( \rho _ { \varphi, \psi  } ) H \omega {\rm d} A+\int _ { \varphi ^{- 1} (E)\setminus K_{\varepsilon}} \Psi ( \rho _ { \varphi, \psi  } ) H \omega {\rm d} A .
\end{align*}
For $z\notin K_{\varepsilon}$. By the non-decreasing of $\Psi$ and \eqref{eq: remark 1}, there exists a constant $C_{\varepsilon}>0$ such that 
\begin{align*}
  \int _ { \varphi ^{- 1} (E)}\Psi ( \rho _ { \varphi, \psi  } ) H  \omega{\rm d} A  & \leq\int _ {\varphi^{- 1} (E) } \chi_{ K_{\varepsilon}}\Psi ( \rho _ { \varphi, \psi  } )H \omega {\rm d} A  +C_{\varepsilon}\int _ { \varphi ^{- 1} (E)}\Psi(M_{\varphi,\psi})H \omega {\rm d} A.
\end{align*}
The same estimate holds with $\varphi$ replaced by $\psi$. 
Let
$$\nu_{\varepsilon}(E) =\int _ {\varphi^{- 1} (E) } \chi_{ K_{\varepsilon}}\Psi ( \rho _ { \varphi, \psi  } ) H \omega {\rm d} A +\int _ {\psi^{- 1} (E) } \chi_{ K_{\varepsilon}}\Psi ( \rho _ { \varphi, \psi  } ) H  \omega{\rm d} A .$$
Therefore
\begin{align}\label{eq:measure-inequality}
  \nu(E) & \leq\nu_{\varepsilon}(E) +C_{\varepsilon}\mu (E).
\end{align}
Let $H_{1}=\chi_{ K_{\varepsilon}}H\leq1 $. So
$$\sup_{z\in\mathbb{D}}H_{1}(z)R_{\varphi}(z)\leq\sup_{z\in K_{\varepsilon}}R_{\varphi}(z)\leq\sup_{z\in K_{\varepsilon}}(R_{\varphi}(z)+R_{\psi}(z))\leq\varepsilon.$$
Similarly,
$$\sup_{z\in\mathbb{D}}H_{1}(z)R_{\psi}(z)\leq\varepsilon.$$
Let $ \{f_{n}\}$ be an arbitrary sequence in $A^{\Psi}_{\omega}(\mathbb{D})$ such that $\sup_{n}\|f_{n}\|_{A^{\Psi}_{\omega}(\mathbb{D})}\leq1$ and $f_{n}\to0$ uniformly on compact sets of $\mathbb{D}$. By Lemma \ref{le: Orlicz-Carleson}, there exists a constant $C_{1}>0$ and $\delta>0$, independent of $\varepsilon$, such that
\begin{align*}
   \int _ {\mathbb{D}}\Psi(|f_{n}\circ \varphi|) H_{1} \omega {\rm d} A & \leq C_{1}\varepsilon^{\delta}\int _ {\mathbb{D}}\Psi(|f_{n}|)\omega {\rm d} A \leq C_{1}\varepsilon^{\delta}
\end{align*}
and $$\int _ {\mathbb{D}}\Psi(|f_{n}\circ\psi|)H_{1} \omega {\rm d} A\leq C_{1}\varepsilon^{\delta}.$$
Since  $0\leq \rho<1$ and $\Psi$ is non-decreasing, it follows from the definition of $\nu_{\varepsilon} $ that
\begin{align*}
  \int _ {\mathbb{D}}\Psi(|f_{n}|) {\rm d}\nu_{\varepsilon} 
   & =\int _ {\mathbb{D}}\Psi(|f_{n}\circ \varphi|) \chi_{ K_{\varepsilon}}\Psi ( \rho _ { \varphi,\psi } )H \omega {\rm d} A+\int _ {\mathbb{D}}\Psi(|f_{n}\circ\psi|)\chi_{ K_{\varepsilon}}\Psi ( \rho _ { \varphi,\psi } )H \omega {\rm d} A\\
   & =\int _ {\mathbb{D}}\Psi(|f_{n}\circ \varphi|) \Psi ( \rho _ { \varphi,\psi } )H_{1} \omega {\rm d} A+\int _ {\mathbb{D}}\Psi(|f_{n}\circ\psi|)\Psi ( \rho _ { \varphi,\psi } )H_{1} \omega {\rm d} A\\
   & \leq 2\Psi(1)C_{1}\varepsilon^{\delta}.
\end{align*}
Since $\mu$ is a vanishing Carleson measure for $A_{\omega}^{\Psi}(\mathbb{D})$,  we have
$$\lim_{n\to\infty}  \int _ {\mathbb{D}}\Psi(|f_{n}|) {\rm d}\mu=0. $$
It follows from \eqref{eq:measure-inequality} that
\begin{align*}
  \int _ {\mathbb{D}}\Psi(|f_{n}|) {\rm d}\nu & \leq \int _ {\mathbb{D}}\Psi(|f_{n}|) {\rm d}\nu_{\varepsilon}+C_{\varepsilon}\int _ {\mathbb{D}}\Psi(|f_{n}|) {\rm d}\mu \\
   & \leq 2\Psi(1)C_{1}\varepsilon^{\delta}+C_{\varepsilon}\int _ {\mathbb{D}}\Psi(|f_{n}|) {\rm d}\mu.
\end{align*} 
For each fix $\varepsilon>0$, letting $n\to\infty$ yields
$$\limsup_{n\to\infty} \int _ {\mathbb{D}}\Psi(|f_{n}|) {\rm d}\nu \leq 2\Psi(1)C_{1}\varepsilon^{\delta}.$$
Since $\varepsilon>0$ is arbitrary. Taking the limit $\varepsilon\to0$, we obtain  $\lim_{n\to\infty}\int _ {\mathbb{D}}\Psi(|f_{n}|) {\rm d}\nu=0$.
By \cite[Theorem 9.4]{KR}, it follows that $\lim_{n\to\infty}\| f_{n}\|_{L^{\Psi}_{\nu}(\mathbb{D})}=0$. This implies that the embedding $I_{d} :A^{\Psi}_{\omega}(\mathbb{D})\rightarrow L^{\Psi}_{\nu}(\mathbb{D})$ is compact, so $\nu$ is a vanishing Carleson measure for $A_{\omega}^{\Psi}(\mathbb{D})$.

\end{proof}

We now turn to the proof of Theorem 1.2.

\textbf{Proof of Theorem 1.2}.
The implication ${\rm (b)}\Rightarrow{\rm (c)}$ follows directly from Lemma \ref{le: vanishing-Carleson-measure}, while ${\rm (c)}\Rightarrow{\rm (b)}$ follows from the boundedness of  $R _ { i }$ for all $i\in \Lambda _ { N }$.

Proof of ${\rm (a)}\Rightarrow{\rm (b)}$.
By the definition of $\mu_{\Psi,\omega}$, $\mu_{\Psi,\omega}$ is a vanishing Carleson measure for $A_{\omega}^{\Psi}(\mathbb{D})$ if and only if $\mu_{\Psi,\eta,\omega}$ is a vanishing Carleson measure for each $\eta \in \mathscr { P } _ { N }$. Thus, it suffices to prove that $\lim_{|z|\to1}\frac{\mu_{\Psi,\eta,\omega}(E_{r}(z))}{\omega(E_{r}(z))}=0$.

 For any Borel set $E\subset\mathbb{D}$, the definition of $\mu_{\Psi,\eta_{i},\omega}$ gives
\begin{align*}
  \mu _{\Psi,\eta,\omega} (E) =& \sum _ { i = 1 } ^{N} | s _ { i } ^{\eta} |  \mu _ { \Psi,\eta _ { i },\omega } (E)=\sum _ { i = 1 } ^{N-1} | s _ { i } ^{\eta} |  \mu _ { \Psi,\eta _ { i },\omega } (E)+| s _ { N } ^{\eta} |  \mu _ { \Psi,\eta _ { N},\omega } (E) \\
   =& \sum _ { i = 1 } ^{N-1} \bigg(\int  _{\varphi _ { \eta _ { i } } ^{- 1} (E) \cap G _ { \eta } } | s _ { i } ^{\eta} |\Psi(M_{\eta _ { i },\eta _ { i+1 }})\omega{\rm d} A +\int  _{\varphi _ { \eta _ { i+1 } } ^{- 1} (E) \cap G _ { \eta } }| s _ { i } ^{\eta} |\Psi( M_{\eta _ { i },\eta_ { i+1 }}) \omega{\rm d} A\bigg)\\
   &+\int _ {   \varphi _ { \eta_ { N } } ^{- 1} (E) \cap G _ { \eta }} | s _ { N } ^{\eta} |\Psi( R _ { \eta _ { N } } ) \omega {\rm d} A .
\end{align*}
If $| s _ { i } ^{\eta} |=0$ for some $1\leq i\leq N$, the corresponding term vanishes. Thus, it suffices to consider indices $i$ for which $| s _ { i } ^{\eta} |>0$. It follows from \eqref{eq: define-Q} that $Q _ { \eta}\geq| s _ { i } ^{\eta} |  M _ { \eta _ { i } , \eta _ { i + 1 } }$ and  $Q _ { \eta }\geq | s _ { N } ^{\eta} |  R _ { \eta _ { N } }$. Combining these inequalities with \eqref{eq: remark 1}, we obtain
\begin{align*}
  \mu _{\Psi,\eta,\omega} (E) &\lesssim\sum _ { i = 1 } ^{N-1} \bigg(\int  _{\varphi _ { \eta _ { i } } ^{- 1} (E)\cap G _ { \eta } } \Psi(Q _ { \eta })\omega{\rm d} A +\int  _{\varphi _ { \eta _ { i+1 } } ^{- 1} (E) \cap G _ { \eta } }\Psi(Q _ { \eta})\omega {\rm d} A\bigg)+\int _ {   \varphi _ {\eta _ { N } } ^{- 1} (E) \cap G _ { \eta }} \Psi(Q _ { \eta }) \omega {\rm d} A \\
  & \leq 2\sum _ {i = 1 } ^{N} \int _ { \varphi _ { i } ^{- 1} (E) \cap G _ { \eta } } \Psi(Q _ { \eta }) \omega  {\rm d} A.
\end{align*}
Fix $r\in (r_{0}(\omega),1)$. For each $z\in \mathbb { D }$, set $E=E_{r}(z)$, we obtain
\begin{align}\label{eq:mu-Psi-estimate}
  \mu _{\Psi,\eta,\omega} (E_{r}(z)) & \lesssim2\sum _ { i = 1 } ^{N} \int _ { \varphi _ { i } ^{- 1} (E_{r}(z)) \cap G _ {\eta} } \Psi(Q _ { \eta })  \omega {\rm d} A.
\end{align}
We next estimate $\Psi(Q _ { \eta }) $ on $ \varphi _ { i } ^{- 1} ( E_{r}(z)) \cap G _ { \eta }$.
Let $\eta \in \mathscr{P}_{N}$. It follows from \eqref{eq: define-G} that
$$Q _ { \eta } (\xi)\leq Q _ { \tau } (\xi) \qquad (\xi \in G _ {\eta })$$
for any $\tau\in \mathscr{P}_{N}$. Taking the product over all $\tau\in \mathscr{P}_{N}$, we obtain
$$( Q _ { \eta } (\xi) ) ^{N !}=\Pi_{\tau\in \mathscr{P}_{N}}Q _ { \eta }(\xi)\leq \Pi_{\tau\in \mathscr{P}_{N}}Q _ { \tau } (\xi)=Q (\xi).$$
Let $$ \mathbb {D}^{t}=\{z\in\mathbb{D}:t<|z|<1\}.$$
Since $T$ is compact on $A^{\Psi}_{\omega}(\mathbb{D})$, Theorem \ref{th: Julia-Caratheodory-type} implies that $\lim\limits_{|\xi|\to1}Q(\xi)=0$.
Thus, for any $\varepsilon > 0$, there exists $t \in (0,1)$ such that $Q (\xi)<\varepsilon^{N !}$ for $|\xi|>t$.
It follows that
\begin{align}\label{eq:Q-eta-estimate}
  Q _ { \eta} (\xi) \leq ( Q (\xi) ) ^{\frac { 1} { N ! } } & < \varepsilon,   \,\,\,\,\,\,\,\xi \in G _ { \eta } \cap \mathbb { D } ^{t}.
\end{align}
Choose $a \in (0,1)$  such that
$$\varphi _ { i } ( t \mathbb { D } ) \subset a\mathbb { D } \,\,\,\,\,\text{or equivalently}\,\,\,\,\,\, \varphi _ { i } ^{- 1} ( \mathbb { D } ^{a } ) \subset \mathbb { D } ^{t}$$
for $1\leq i\leq N$.

For $a\in(0,1)$, it follows from \eqref{eq:Euclidean disk} that there exists $r_1\in(0,1)$ such that $E_{r}(z) \subset \mathbb{D}^{a}$ for $z\in\mathbb D^{r_1}$.
Thus $$\varphi _ { i } ^{- 1}(E_{r}(z))\subset \varphi _ { i } ^{- 1}(\mathbb{D}^{a})\subset\mathbb { D } ^{t},  \qquad z\in\mathbb D^{r_1}.$$
This implies that
$$\varphi _ { i } ^{- 1}(E_{r}(z))\cap G _ { \eta }\subset\mathbb { D } ^{t}\cap G _ { \eta }.$$
Combining this with \eqref{eq:Q-eta-estimate}, we obtain
 $$Q _ { \eta } (\xi)< \varepsilon,  \qquad  \xi \in \varphi _ { i } ^{- 1}(E_{r}(z))\cap G _ { \eta}, \qquad z\in\mathbb D^{r_1}$$
for $1\leq i\leq N$. Hence, by the non-decreasing of $\Psi$ and \eqref{eq:mu-Psi-estimate}, we obtain
\begin{align*}
    \frac{\mu _{\Psi,\eta,\omega} (E_{r}(z))}{\omega(E_{r}(z))}  & \lesssim  2\Psi(\varepsilon)\sum _ { i= 1 } ^{N} \frac{\int _ { \varphi _ { i } ^{- 1} ( E_{r}(z)) } \omega {\rm d} A}{\omega(E_{r}(z))}\\
   & =2\Psi(\varepsilon)\sum _ { i = 1 } ^{N} \frac{\mu_{ \varphi _ { i },\omega} ( E_{r}(z)))}{\omega(E_{r}(z))} .
\end{align*}
Since $\omega\in \mathcal{D}$, $C_{\varphi}$ is bounded on $A^{p}(\omega)$ and $\mu_{\varphi, \omega}$ is a Carleson measure for $A^{p}(\omega)$. By Remark 2, $\mu_{\varphi, \omega}$ is also a Carleson measure for $A_{\omega}^{\Psi}(\mathbb D)$. By Lemma \ref{le:Carleson measure}, let $\varepsilon\to0$, we have 
 $$\lim_{|z|\to1}\frac{\mu_{\Psi,\eta,\omega}(E_{r}(z))}{\omega(E_{r}(z))}=0 .$$
Therefore, $\mu_{\Psi,\eta,\omega}$ is a vanishing  Carleson measure for  $A_{\omega}^{\Psi}(\mathbb D)$,
which implies that $\mu_{\Psi,\omega}$ is also a vanishing  Carleson measure $A_{\omega}^{\Psi}(\mathbb D)$.

Proof of ${\rm (b)}\Rightarrow{\rm (a)}$. To prove $T$ is compact on $A^{\Psi}_{\omega}(\mathbb{D})$, let $\{f_{n}\}$ be an arbitrary sequence in $A^{\Psi}_{\omega}(\mathbb{D})$ such that $\sup_{n}\|f_ {n}\|_{A^{\Psi}_{\omega}(\mathbb{D})} \leq1$ and $f_{n}\rightarrow0$ uniformly on compact subsets of $\mathbb{D}$. By Lemma \ref{le:compactness}, it suffices to show that $\|Tf_ {n}\|_{A^{\Psi}_{\omega}(\mathbb{D})}\rightarrow0$.
Fix $r \in (0,1)$.
Since $r\mathbb{D}$ is a compact subset of $\mathbb{D}$,
$\{f_{n}\}$ converges uniformly to $0$ on $r\mathbb{D}$. Noting that
$$ \mathbb { D }= \bigcup _ { \eta \in \mathscr { P } _ { N } } G _ {\eta}, $$
we have
$$\mathbb{D}^{r}=\bigg(\bigcup _ { \eta \in \mathscr { P } _ { N } } G _ { \eta }\bigg)\cap \mathbb{D}^{r}= \bigcup _ { \eta \in \mathscr { P } _ { N } } \bigg(G _ { \eta }\cap \mathbb{D}^{r}\bigg)=\bigcup _ { \eta \in \mathscr { P } _ { N } } G_{\eta}^{r}.$$
Hence
\begin{align*}
 \limsup _ { n \to \infty } \int _ { \mathbb { D } } \Psi(| T f _ { n } |)  \omega  {\rm d} A & = \limsup _ { n \to \infty } \int _ {\mathbb { D }^{r} } \Psi(| T f _ { n } |)  \omega  {\rm d} A+\limsup _ { n \to \infty } \int _ { r\mathbb{D} }\Psi(| T f _ { n } |)  \omega {\rm d} A \\
   & \leq \limsup _ { n \to \infty } \sum _ {\eta \in \mathscr { P } _ { N }}\int_{G _ { \eta } ^{r} }  \Psi(| T f _ { n } |)  \omega {\rm d} A .
\end{align*}

Let $$T_{1}= \sum _ { i = 1 } ^{N - 1} s _ { i } ^{\eta} T _ { \eta _ { i } , \eta_ { i + 1 } },\,\,\,\,T_{2}=  s _ { N } ^{\eta} T _ { \eta _ {N}}. $$
By the non-decreasing of $\Psi$ and \eqref{eq: subadditivity}, we have
$$\Psi(|T_{1}f _ { n }|)\leq\Psi(\sum _ { i = 1 } ^{N - 1} |s _ { i } ^{\eta}| |T _ { \eta_ { i } ,\eta _ { i + 1 } } f _ { n }|)\lesssim\sum _ { i = 1 } ^{N - 1}\Psi( |s _ { i } ^{\eta}| |T _ { \eta_ { i } , \eta _ { i + 1 } } f _ { n }|).$$
For $0\leq|s _ { i } ^{\eta}|\leq1$, the convexity of $\Psi$ gives
$$\Psi( |s _ { i } ^{\eta}| |T _ { \eta _ { i } , \eta_ { i + 1 } } f _ { n }|)\leq|s _ { i } ^{\eta}|\Psi(  |T _ { \eta _ { i } , \eta_ { i + 1 } } f _ { n }|).$$
For $|s _ { i } ^{\eta}|>1$, the $\Delta_{2}$-condition yields
$$\Psi( |s _ { i } ^{\eta}| |T _ { \eta _ { i } , \eta _ { i + 1 } } f _ { n }|)\lesssim|s _ { i } ^{\eta}|^{p}\Psi(  |T _ { \eta _ { i } , \eta _ { i + 1 } } f _ { n }|)=|s _ { i } ^{\eta}|^{\max\{p,1\}-1} |s _ { i } ^{\eta}|\Psi(  |T _ { \eta _ { i } , \eta_ { i + 1 } } f _ { n }|).$$
Therefore
$$\Psi( |s _ { i } ^{\eta}| |T _ { \eta _ { i } , \eta_ { i + 1 } } f _ { n }|)\lesssim\max\{1,|s _ { i } ^{\eta}|^{\max\{p,1\}-1}\}|s _ { i } ^{\eta}|\Psi(  |T _ { \eta _ { i } , \eta _ { i + 1 } } f _ { n }|)\,\,\,\,\, \mbox{for}\,|s _ { i } ^{\eta}|\geq0, \,1\leq i< N.$$
It follows that
\begin{align*}
\int_{G _ { \eta } ^{r} }  \Psi(| T_{1} f _ { n } |) \omega  {\rm d} A &\lesssim \sum _ { i = 1 } ^{N - 1} |s _ { i } ^{\eta}| \int_{G _ { \eta } ^{r} } \Psi(  |T _ { \eta_ { i } , \eta _ { i + 1 } } f _ { n }|)\omega {\rm d} A \\
       & = \sum _ { i = 1 } ^{N - 1} |s _ { i } ^{\eta}|(I _ { 1n } ^{\eta, \varepsilon, i} + I _ { 2n } ^{\eta , \varepsilon , i} ),
\end{align*}
where
\begin{align}\label{eq: I-1n}
  I _ { 1n } ^{\eta, \varepsilon , i}  & =\int _ {G _ { \eta } ^{r} \cap \{ M _ { \eta _ { i } , \eta _ { i + 1 } } \geq \varepsilon\} }\Psi(|T _ { \eta _ { i } , \eta _ { i + 1 } }f _ { n } |)  \omega {\rm d} A,
\end{align}
 \begin{align}\label{eq: I-2n}
 I _ { 2n } ^{\eta, \varepsilon , i} &=\int_{G _ {\eta } ^{r} \cap \{ M _ {\eta _ { i } , \eta _ { i + 1 } } < \varepsilon \}}  \Psi(| T _ { \eta _ { i } , \eta _ { i + 1 } } f _ { n } |) \omega {\rm d} A.
\end{align}
By the same argument, we  also obtain
$$\Psi(|T_{2}f _ { n }|)= \Psi(| s _ { N } ^{\eta} T _ { \eta _ {N}} f _ { n }|)\lesssim|s _ { N } ^{\eta}|\Psi(  |T _ { \eta _ { N } } f _ { n }|).$$
Hence
\begin{align*}
\int_{G _ { \eta } ^{r} }  \Psi(| T_{2} f _ { n } |)  \omega {\rm d} A & \lesssim|s _ { N } ^{\eta}| \int_{G _ { \eta } ^{r} }  \Psi(  |T _ {\eta _ { N } } f _ { n }|) \omega {\rm d} A \\
 & = |s _ { N } ^{\eta}|(I _ { 1n } ^{\eta , \varepsilon, N} + I _ { 2n } ^{\eta, \varepsilon , N} ),
\end{align*}
where
\begin{align}\label{eq: I-1N}
  I _ { 1n } ^{\eta ,\varepsilon , N}  &=\int _ {G _ { \eta } ^{r} \cap \{ R _ { \eta _ { N }  } \geq \varepsilon\} }\Psi(|T _ { \eta _ { N } }f _ { n } |) \omega {\rm d} A,
\end{align}
\begin{align}\label{eq: I-2N}
I _ { 2n } ^{\eta , \varepsilon , N} &=\int_{G _ {\eta} ^{r} \cap \{ R _ {\eta _ { N } } < \varepsilon \}} \Psi(| T _ {\eta _ { N } } f _ { n } |)  \omega {\rm d} A  .
\end{align}
Thus, for $T=T_{1}+T_{2}$, we deduce that
\begin{align*}
  \int_{G _ { \eta } ^{r} }  \Psi(| T f _ { n } |)  \omega {\rm d} A   &\lesssim \sum _ { i = 1 } ^{N } |s _ { i } ^{\eta}| (I _ {1 n } ^{\eta, \varepsilon, i} + I _ {2 n } ^{\eta , \varepsilon , i} ) .
\end{align*}

To estimate \eqref{eq: I-1n}, it follows from the non-decreasing, convexity of $\Psi$ and $\Delta_{2}$-condition that
\begin{align*}
  I _ { 1n } ^{\eta, \varepsilon , i} 
   & \leq \frac{ 1 }{\Psi(\varepsilon)}\int _ {G _ {\eta} }\Psi(|T _ { \eta_ { i } , \eta _ { i + 1 } }f _ { n } |)\Psi(M _ { \eta _ { i } ,\eta _ { i + 1 } }) \omega {\rm d} A \\
   &\lesssim \frac{ 1 }{\Psi(\varepsilon)}\int _ {G _ { \eta } }\bigg(\Psi(|T _ {\eta _ { i } }f _ { n }|)+ \Psi(|T _ {\eta_ { i + 1 } }f _ { n } |)\bigg)\Psi(M _ { \eta _ { i } , \eta _ { i + 1 } }) \omega {\rm d} A.
\end{align*}
For \eqref{eq: I-1N}, by the non-decreasing of $\Psi$, we get
\begin{align*}
  I _ { 1n } ^{\eta , \varepsilon , N} 
   & \leq \frac{ 1 }{\Psi(\varepsilon)}\int _ {G _ {\eta } }\Psi(|T _ { \eta_ { N }  }f _ { n } |)\Psi( R _ { \eta _ { N }  } ) \omega {\rm d}A .
\end{align*}
Hence, for $ 1 \leq i \leq N$, it follows from \eqref{eq: measure} that
\begin{align*}
  \sum _ { i = 1 } ^{N } |s _ { i } ^{\eta}|I _ {1 n } ^{\eta, \varepsilon, i} & =\sum _ { i = 1 } ^{N-1 } |s _ { i } ^{\eta}|I _ { 1n } ^{\eta , \varepsilon, i} + |s _ { N } ^{\eta}|I _ { 1n } ^{\eta , \varepsilon, N} \\
   & \lesssim \frac{ 1 }{\Psi(\varepsilon)}\bigg(\sum _ { i = 1 } ^{N-1 } |s _ {i } ^{\eta}|\int _ {G _ { \eta } }\bigg(\Psi(|T _ { \eta_ { i } }f _ { n }|)+ \Psi(|T _ {\eta_ { i + 1 } }f _ { n } |)\bigg)\Psi(M _ { \eta _ { i } , \eta _ { i + 1 } }) \omega {\rm d} A \\
   &+|s _ { N } ^{\eta}|\int _ {G _ { \eta } }\Psi(|T _ {\eta_ { N }  }f _ { n } |)\Psi( R _ { \eta _ { N }  } ) \omega {\rm d} A \bigg)  \\
   & =\frac{ 1 }{\Psi(\varepsilon)}\int _ {\mathbb{D}}\Psi(|f _ { n } |)  {\rm d} \mu_{\Psi,\eta,\omega}.
\end{align*}
Since $\mu_{\Psi,\omega}$ is a vanishing Carleson measure for $A^{\Psi}_{\omega}(\mathbb{D})$ and $\mu_{\Psi,\omega}= \sum _ { \eta \in \mathscr { P } _ { N } } \mu _{\Psi,\eta,\omega}$, $\mu_{\Psi,\eta,\omega}$ is also a vanishing Carleson measure for $A^{\Psi}_{\omega}(\mathbb{D})$. Note that $\sup\limits_{n}\|f_ {n}\|_{A^{\Psi}_{\omega}(\mathbb{D})}
 \leq1$ and $f_{n}\rightarrow0$ uniformly on compact subsets of $\mathbb{D}$,  it follows from \cite[Theorem 9.4]{KR} that $\lim\limits_{n\rightarrow\infty} \int _ {\mathbb{D}}\Psi(|f _ { n } |)  {\rm d} \mu_{\Psi,\eta,\omega}=0$.
Hence
$$\lim_{n\rightarrow\infty}\sum _ { i = 1 } ^{N } |s _ { i } ^{\eta}|I _ { 1n } ^{\eta , \varepsilon, i}=0.$$

For \eqref{eq: I-2n},
let $K=G _ { \eta } ^{r} \cap \{ M _ { \eta _ { i } ,\eta_ { i + 1 } } < \varepsilon \}\subset\mathbb{D}$, then $\sup\limits_{z\in K} M _ { \eta_ { i } ,\eta_ { i + 1 } } \leq\varepsilon$. By Lemma \ref{le: symbols-difference}, there exists a constant  $h_{\eta}(\varepsilon)=\sum _ { i = 1 } ^{N-1 } |s _ { i } ^{\eta}| h_{i}(\varepsilon) >0$ such that $\lim\limits_{\varepsilon\rightarrow0}h_{\eta}(\varepsilon)=0$ and
\begin{align*}
  \sum _ { i = 1 } ^{N-1 } |s _ { i } ^{\eta}| I _ { 2n } ^{\eta, \varepsilon , i} 
   & = \sum _ { i = 1 } ^{N-1 } |s _ { i } ^{\eta}|\int_{K}\Psi(|T _ { \eta _ { i }}f _ { n } - T _ {\eta _ { i + 1 } }f _ { n } |) \omega{\rm d}A \\
   & \leq \sum _ { i = 1 } ^{N-1 } |s _ { i } ^{\eta}|h_{i}(\varepsilon) \int_{\mathbb{D}}  \Psi(| f _ { n } |)  \omega {\rm d} A  \leq h_{\eta}(\varepsilon) .
\end{align*}
For \eqref{eq: I-2N},
let $H=\chi_{G _ { \eta } ^{r} \cap \{ R _ { \eta_ { N } } < \varepsilon \}}\leq1$, then $\sup\limits_{z\in\mathbb{D}}H(z) R _ { \eta _ { N }}(z)\leq\varepsilon$. By Lemma \ref{le: Orlicz-Carleson}, there exists a constant $C=|s _ {N} ^{\eta}|C_{1}>0$ and $\delta>0$, independent of $\varepsilon$, such that
\begin{align*}
 |s _ {N} ^{\eta}| I_ { 2n } ^{\eta, \varepsilon , N} 
   & =  |s _ {N} ^{\eta}| \int_{\mathbb{D}}  \Psi(| T _ { \eta _ { N }  } f _ { n } |) H \omega {\rm d} A\\
   & \leq |s _ {N} ^{\eta}|C_{1}\varepsilon^{\delta}\int_{\mathbb{D}} \Psi(| f _ { n } |) \omega {\rm d} A\leq C\varepsilon^{\delta}.
\end{align*}
Combining the estimates above, we obtain
\begin{align*}
  \limsup_{n\rightarrow\infty}\int_{G _ { \eta} ^{r} }  \Psi(| T f _ { n } |)  \omega {\rm d} A &\leq \limsup_{n\rightarrow\infty}\sum _ { i= 1 } ^{N } |s _ { i } ^{\eta}| (I _ {1 n } ^{\eta , \varepsilon, i} + I  _ {2 n } ^{\eta, \varepsilon , i} ) \\
   & \leq  h_{\eta}(\varepsilon)+C\varepsilon^{\delta}.
\end{align*}
Taking the limit $\varepsilon\rightarrow0$, we have
$$\limsup_{n\rightarrow\infty}\int_{G _ { \eta } ^{r} }  \Psi(| T f _ { n } |)  \omega {\rm d} A=0.$$
 Hence $$\lim _ { n \to \infty } \int _ { \mathbb { D } } \Psi(| T f _ { n } |)    \omega{\rm d} A=0.$$
It follows from \cite[Theorem 9.4]{KR}  that $\lim_{ n\to \infty}\|Tf_{n}\|_{A^{\Psi}_{\omega}(\mathbb{D})}=0$. Consequently, $T$ is compact on $A^{\Psi}_{\omega}(\mathbb{D})$.
This completes the proof of Theorem 1.2.

\end{document}